%% file: SRL.tex
\documentclass[twoside,a4paper,12pt]{amsart}
\usepackage{amsmath, amssymb, amsfonts, stmaryrd}% ,mathdots?
\usepackage{type1cm}
\usepackage[lmargin=1.3in,rmargin=1.3in,bottom=1.0in,top=1.5in,twoside=False]{geometry}
\usepackage{tensor}
\usepackage{enumerate}
\usepackage{graphicx}
\usepackage[czech,english]{babel}
\usepackage[single]{accents}
\usepackage{mdwtab}
\usepackage[OT1]{fontenc}
\usepackage[all]{xy}
\usepackage{mathtools}
\usepackage{microtype}
\usepackage{fontaxes}
\usepackage[all]{nowidow}
\usepackage{hyperref}
\usepackage[nameinlink]{cleveref}
\usepackage{xcolor}
\hypersetup{
    colorlinks,
    linkcolor={blue!80!black},
    citecolor={blue!80!black},
    urlcolor={blue!80!black}
}
\renewcommand{\epsilon}{\varepsilon}

\newcommand{\shtm}{\,\widetilde{\triangledown}\,}

\usepackage{color}
\definecolor{shadecolor}{gray}{.85}%
\definecolor{tintedcolor}{gray}{.80}%
\usepackage{framed}%
\definecolor{mytintedcolor}{gray}{.95}%
\makeatletter
\newdimen\svparindent
\newcounter{tmpthm}
\newenvironment{mytinted}{%
  \MakeFramed {\FrameRestore}}%
{\endMakeFramed}
{\endlist\end{mytinted}\egroup}
\makeatother
\newlength{\graphshift}

\usepackage{mathrsfs}
\newcommand{\Cc}{\mathscr{C}}
\newcommand{\Dd}{\mathscr{D}}

\newcommand{\strM}{\mathcal{M}}

\newcommand{\N}{\mathbb{N}}

\newcommand{\Q}{\mathbb{Q}}
\newcommand{\R}{\mathbb{R}}
\newcommand{\rw}{{\rm rw}}
\newcommand{\lrw}{{\rm lrw}}
\renewcommand{\phi}{\varphi}
\newcommand{\ERCagreement}{{\begin{minipage}{.63\textwidth}\tiny This paper is part of a project that has received funding from the European Research Council (ERC) under the European Union's Horizon 2020 research and innovation programme (grant agreement No 810115 -- {\sc Dynasnet}). \end{minipage}\hfill\begin{minipage}{.32\textwidth}\includegraphics[width=\textwidth]{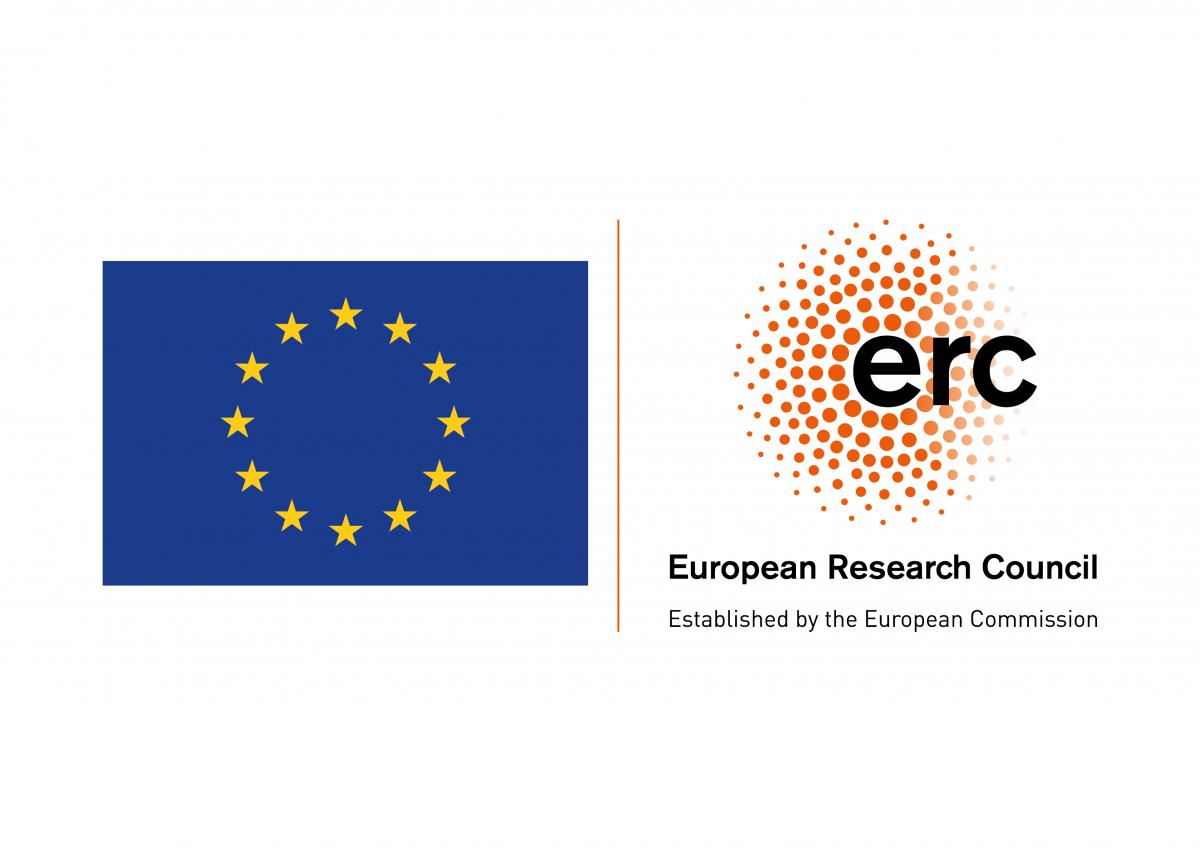}\\\end{minipage}\hfill}}

\newtheorem{theorem}{Theorem}
\newtheorem{proposition}{Proposition}
\newtheorem{observation}{Observation}

\newtheorem{corollary}{Corollary}
\newtheorem{lemma}[theorem]{Lemma}
\newtheorem{claim}{Claim}[theorem]

\theoremstyle{definition}
\newtheorem{definition}[theorem]{Definition}
\newtheorem{example}[theorem]{Example}

\theoremstyle{remark}
\newtheorem{remark}[theorem]{Remark}

\newtheorem{problem}{Problem}
\newtheorem{conjecture}{Conjecture}
\DeclareMathOperator{\dens}{\rm dens}
\def\cqedsymbol{\ifmmode$\lrcorner$\else{\unskip\nobreak\hfil
\penalty50\hskip1em\null\nobreak\hfil$\lrcorner$
\parfillskip=0pt\finalhyphendemerits=0\endgraf}\fi} 

\newcommand{\cqed}{\renewcommand{\qed}{\cqedsymbol}}

\title[Regular partitions of gentle graphs]{Regular partitions of gentle graphs\footnote{\ERCagreement}}
\dedicatory{Dedicated to Endre Szemer\'edi on the occasion of his eightieth birthday.}
\author[Y. Jiang]{Yiting Jiang}
\address{Yiting Jiang\\
Institut de recherche en informatique fondamentale\\
Paris, France
}
\email{yjiang@irif.fr}
\author[J. Ne\v set\v ril]{Jaroslav Ne\v set\v ril}
\address{Jaroslav Ne{\v s}et{\v r}il\\
Computer Science Institute of Charles University (IUUK)\\
Praha, Czech Republic}
\email{nesetril@iuuk.mff.cuni.cz}
\author[P. Ossona de Mendez]{Patrice {Ossona de Mendez}}
\author[S. Siebertz]{Sebastian Siebertz}
\address{Sebastian Siebertz\\University of Bremen,
 Germany}
\email{siebertz@uni-bremen.de}
\address{Patrice~Ossona~de~Mendez\\
Centre d'Analyse et de Math\'ematiques Sociales (CNRS, UMR 8557),
Paris, France\\
  and     Computer Science Institute of Charles University (IUUK),
  Praha, Czech Republic
}
 \email{pom@ehess.fr}
\begin{document}

\input{abstract}

\maketitle
\setcounter{tocdepth}{1}
\tableofcontents

\input{intro}

\part{Preliminaries and survey of some regularity lemmas}
\input{logic}
\input{graphs}
\input{sparsity}
\input{regularity}

\part{Regularity for gentle graphs}
\input{set}

\input{covers}
\input{cographs}

%\input{rw}
\input{nd}
\addtocontents{toc}{\vspace{1\baselineskip}}
\input{conclusion}

\section*{Acknowledgments}
The authors are indebted to O-joung Kwon, who pointed out an error in an earlier version of this paper. 

\bibliographystyle{amsplain}
\bibliography{ref}
\end{document}

%% file: abstract.tex
\begin{abstract} 
  Szemer\'edi's Regularity Lemma is a very useful tool of extremal
  combinatorics. %and theoretical computer science.
  Recently, several refinements of this seminal result were obtained
  for special, more structured classes of graphs. We survey these
  results in their rich combinatorial context. In particular, we
  stress the link to the theory of (structural) sparsity, which leads
  to alternative proofs, refinements and solutions of open problems.
  It is interesting to note that many of these classes present
  challenging problems. Nevertheless, from the point of view of
  regularity lemma type statements, they appear as ``gentle'' classes.
\end{abstract}

%% file: intro.tex
\section*{Introduction}

\subsection*{Szemer\'edi's Regularity Lem\-ma} Szemer\'edi's Regularity
Lem\-ma~\cite{sz} is a very useful tool in extremal graph
theory. %and theoretical computer science.
Informally, the lemma states that the vertices of every sufficiently
large graph can be partitioned into a bounded number of parts so that
the edges between almost all pairs of different parts behave in a
sense like random graphs. Let us give the formal definitions.

\begin{definition}
  Let $G$ be a graph and let $A,B\subseteq V(G)$ be two disjoint
  non-empty subsets of vertices. We write $E(A,B)$ for the set of
  edges with one end in $A$ and the other end in $B$. We define the
  \emph{density} of the pair $(A,B)$ as
\[\dens(A,B)\coloneqq \frac{|E(A,B)|}{|A| |B|}.\]
\end{definition}

\begin{definition}
  Let $\epsilon>0$, let $G$ be a graph and let $A,B\subseteq V(G)$ be
  two disjoint non-empty subsets of vertices. We call the pair $(A,B)$
  \emph{$\epsilon$-regular} if, for all subsets $A'\subseteq A$ and
  $B'\subseteq B$ with $|A'|\geq \epsilon|A|$ and
  $|B'|\geq \epsilon |B|$, we have
\[|\dens(A',B')-\dens(A,B)|\leq \epsilon.\]
\end{definition}

This uniform distribution of edges is typical in random bipartite
graphs.

\begin{definition}
  A partition
  $V=V_1\mathop{\dot\cup}V_2\mathop{\dot\cup}\ldots
  \mathop{\dot\cup}V_k$
  of a set into disjoint parts is called an \emph{equipartition} if
  $||V_i|-|V_j||\leq 1$ for $1\leq i<j\leq k$.
\end{definition}

\begin{theorem}[Szemer\'edi's Regularity Lemma~\cite{sz}] For every
  real $\epsilon>0$ and integer~$m\geq 1$ there exists two integers
  $M$ and $N$ with the following property.  For every graph $G$ with
  $n\geq N$ vertices there exists an equipartition of the vertex set
  into $k$ classes $V_1,\ldots, V_k$, $m\leq k\leq M$, such that all but at most
  $\epsilon k^2$ of the pairs~$(V_i,V_j)$ are $\epsilon$-regular.
\end{theorem}
Szemer\'edi's Regularity Lemma is a high level approximation scheme
for large graphs and has many applications.  It may be seen as an
essential theoretical justification for the introduction and the study
of the so-called {\em stochastic block model} in statistics
\cite{holland1983stochastic}. In this sense, the presence of densities
brings this result closer to a random graph model than to a graph
approximation.

Szemer\'edi's Regularity Lemma is not only a fundamental result in
graph theory, it also %is not an isolated graph theory
% result. Not only it has been found in
% the context of a deep number theory
% problem \cite{szemeredi1975sets} but
% it
led to extensions and new proofs in several other mathematical areas
such as analysis \cite{analyst}, information theory
\cite{tao2005szemer}, number theory \cite{green2010arithmetic},
%spectral theory {Tao}, 
hypergraphs \cite{Rodl2007, Gowers2007, ElekSze} and relational
structures \cite{aroskar2014limits}, algebra \cite{tao2015expanding}
and algebraic geometry \cite{fox2012overlap}. There are countless
applications of the regularity lemma.  One of the key uses of the
lemma is to transfer results from random graphs, which are much easier
to handle, to the class of all graphs of a given edge density. We
refer to the papers~\cite{alon1994algorithmic, komlos1996szemeredi}
for extensive background on the applications of the regularity lemma.

As shown by Conlon and Fox in~\cite{conlon2012bounds}, the exceptional
pairs cannot be avoided in the statement of Szemer\'edi's Regularity
Lemma, as witnessed by the example of half-graphs. A half-graph is a
bipartite graph with vertices $a_1,\ldots, a_n, b_1,\ldots, b_n$ for
some integer $n\geq 1$ and edges $\{a_i,b_j\}$ for
$1\leq i\leq j\leq n$. The bound $M$ for the number of parts in the
partition of the graph is very large: it has to grow as a tower of
$2$'s of height $\Omega(\epsilon^{-2})$, as proved by Fox and Lov\'asz
in~\cite{fox2014tight}, extending earlier results of Gowers
\cite{gowers1997lower}.  

Consequently, it is natural to ask if restrictions on the graph being
partitioned might result in a stronger form of regularity.  Such
restricted versions of the regularity lemma were established for
example in \cite{ackerman2017stable, alon2005crossing,
  chernikov2015regularity, chernikov2016definable, fox2016polynomial,
  malliaris2014regularity, pillay2013remarks, simon2016note,
  starchenko2016nip}. These results establish for example a polynomial
number of parts, stronger forms of regularity, the absence of
exceptional pairs, etc., in restricted graph classes. 
Results of this kind will be simply
called ``regularity lemmas'' in this paper.

\medskip
In the first part of this work we survey these results in their rich
combinatorial context. After this survey part, the stage is set for
our further study of regularity properties of low complexity graph classes.

% Such context is provided by sparse structures in many forms.

\subsection*{Sparsity and low complexity classes}

Szemer\'edi's regularity lemma is useful only for dense graphs. For
sparse graphs, i.e.\ graphs with a sub-quadratic number of edges, it
becomes trivial, as every balanced partition into a suitable constant
number of pieces is $\epsilon$-regular. Nevertheless, various 
regularity lemmas for sparse graphs exist, see
e.g.~\cite{kohayakawa1997szemeredi, kohayakawa2003szemeredi, Rodl2015,
  gerke2005sparse,scott2011szemeredi}.  In this paper, 
  sparse graphs are considered in the context of 
  combinatorially defined classes of graphs,
which recently formed a very active area. This is referred to briefly
as ``sparsity'', with key notions, such as bounded expansion, nowhere
denseness, quasi-wideness, etc.\ (see e.g.~\cite{nevsetvril2012sparsity}),
as well as notions from geometric and structural graph theory.  All of
this will be reviewed below in \Cref{sec:graph}. 
%In fact, some of the
%regularity lemmas established for restricted graph classes do make
%sense for sparse graphs, or do not even hold, e.g.\ we will prove in
%\Cref{sec:nd} that there exist nowhere dense graph classes that do not
%allow $\epsilon$-nice partitions for any $\epsilon<1$.  However, this
%will not be the main focus of this paper.
  
\smallskip
We will focus on numerous questions that arise
when we consider dense graphs that are constructed from sparse graphs,
e.g.\ map graphs, which are induced subgraphs of squares of planar
graphs. The operations of taking a graph power and taking induced
subgraphs are special cases of \emph{logical transduction} 
and of \emph{logical interpretations}, which are studied in model theory. 
As a second example, graphs of bounded cliquewidth, or
equivalently, bounded rankwidth, are first-order transductions of
tree-orders. In this sense, model theory offers a very convenient 
way to construct graphs from other well behaved structures 
via interpretations and transductions.  
Also, some of the stronger forms of the regularity lemma are based
on model theoretic notions, e.g.\ for graphs defined in distal
theories~\cite{chernikov2015regularity,simon2016note} or graphs with a
stable edge relation~\cite{malliaris2014regularity}.

\smallskip
One of the essential tools in the study of sparse classes are the
so-called \emph{low treedepth decompositions} \cite{Taxi_tdepth, 
  nevsetvril2008grad, nevsetvril2012sparsity, Japan04}, referred to as
\emph{$p$-covers by classes with bounded treedepth} in this paper (see
\Cref{sec:2cov}). This tool has been extended to {\em structurally
  bounded expansion classes}, that is, to graph classes obtained as
transductions of classes with bounded expansion, which turn out to be
characterized by the existence of $p$-covers by classes with bounded
shrubdepth \cite{SBE_drops}. This type of decomposition has been
extended to graphs with linear rankwidth (which are $2$-covered by
classes with bounded embedded shrubdepth \cite{SODA_msrw}) and to
classes $p$-covered by classes with bounded rankwidth
\cite{kwon2020low}. We will see in \Cref{sec:2cov} that there is a
nice interplay between the notions of graph regularity and the
existence of $2$-covers by classes of graphs with small
complexity. For the first time, we consider base classes consisting of
{\em embedded $m$-partite cographs}, which generalize bounded
treedepth, bounded shrubdepth, and bounded embedded shrubdepth. Then
we consider the more general case of a base class with bounded
rankwidth.

% In the second part of this work we consider constructions from
% sparse graph classes to low complexity classes, which are tame from
% the model theoretic point of view, and study their regularity
% properties.

\medskip
\subsection*{Our results}

In \Cref{sec:setdef} we introduce set-defined classes, which are
semi-algebraic and have bounded order-dimension. As a consequence,
they enjoy both stable regularity (\Cref{thm:stable_reg}) and
semi-algebraic regularity (\Cref{thm:sa}). We give important examples
of set-defined classes and prove that set-defined classes are a
dense analog of degenerate classes (\Cref{thm:conj1}).

\smallskip
In \Cref{sec:2cov} we consider $2$-covers of a class by another class
and how these $2$-covers transport properties like being
distal-defined, semi-algebraic, set-defined, having bounded
VC-dimension or bounded order-dimension, as well as the exis\-tence of
polynomial $\epsilon$-nice partitions (just as in the distal
regularity lemma \Cref{thm:distal_reg}).
%Such covers play a key role
%in the theory of bounded expansion and nowhere dense classes and were
%recently used in the study of low complexity
%classes~\cite{SBE_drops,kwon2020low}.

\smallskip
In \Cref{sec:cog}, we show that classes $2$-covered by a class of
embedded $m$-partite cographs are semi-algebraic, and hence satisfy
the semi-algebraic regularity lemma.  Moreover, we give an explicit
construction for the construction of an $\epsilon$-nice partition with
explicit bound for the number of parts (\Cref{cor:reg_emb_2cov}) in
the style of the distal regularity lemma.
 
% \smallskip
%In \Cref{sec:rw}, we prove that every class $2$-covered by a class
%with bounded rankwidth is semi-algebraic, hence satisfies the
%regularity lemma for semi-algebraic classes (\Cref{cor:rw-2cov-reg}).

\smallskip
In \Cref{sec:nd}, we study regularity properties of nowhere dense
classes, and characterize nowhere dense classes in terms of the regularity properties of the $d$-powers of the
graphs in the class (\Cref{thm:ndreg}). On a negative side, we prove
that there exists a nowhere dense class that not only is not
distal-defined, but also does not allow $\epsilon$-nice partitions for
any $\epsilon<1$ (\Cref{cor:nd-noreg}).

\smallskip
Summarizing we provide many new regularity lemmas for the sparse classes defined by the means of combinatorial and model theoretical  tools.

%% file: logic.tex
\section{Model theory background}
\subsection{Structures}

A language $L$ is a set of function symbols, relation symbols and
constant symbols. To each function symbol $f\in L$ and each relation
symbol $R\in L$ we associate an arity. Let $L$ be a language. An $L$-structure $\strM$ consists
of a nonempty set $M$, called the universe or domain of $\strM$, a
function $f_\strM\colon M^k\rightarrow M$ for each $k$-ary function
symbol $f\in L$, a relation \mbox{$R_\strM\subseteq M^k$} for each
$k$-ary relation symbol $R\in F$ and an element $c_\strM\in M$ for
each constant symbol $c\in L$.  The functions $f_\strM, R_\strM$ and
$c_\strM$ are called the interpretations of $f,R$ and $c$ in
$\strM$. If no confusion can arise we do not distinguish between a
symbol and its interpretation.

\smallskip
\begin{example}
  The standard language for real closed fields is
  $L_{\text{RCF}}=\{+,\cdot,<,0,1\}$, where $+$ and $\cdot$ are
  binary function symbols, $<$ is a binary relation symbol and $0,1$
  are constant symbols. The %$L_{\mathit{RCF}}$-structure
  field of real numbers $(\R,+,\cdot,<,0,1)$, where $+,\cdot,<,0,1$ are
  interpreted as usual, is a prototypical model of the theory of real
  closed fields.
\end{example}

\smallskip
A graph can be seen as a structure over the language
$L_{\text{graph}}=\{E\}$, where $E$ is a binary relation symbol. For
this, we identify a symmetric and irreflexive binary relation with a
set of undirected edges.

\medskip
\subsection{First-order logic, interpretations and transductions}

We use standard first-order logic and refer to~\cite{hodges1993model}
for more background.  When the language~$L$ is clear from context, we
shall use the term {\em formula} for an {\em $L$-formula}, that is a first-order formula in the language $L$.
% We can formally define the notion of $L$-formula (that is of a
% first-order formula based the language $L$) as follows: fix an
% infinite set $\textsc{Var}$ of first-order variables and first
% define the set of \emph{$L$-terms} as the smallest set such that all
% variables $x\in \textsc{Var}$ are $L$-terms, all constant symbols
% $c\in L$ are $L$-terms, and if $f\in L$ is a $k$-ary function symbol
% and $t_1,\ldots, t_k$ are $L$-terms, then also $f(t_1,\ldots, t_k)$
% is an $L$-term.  The set of \emph{atomic $L$-formulas} is the
% smallest set such that if $s$ and $t$ are $L$-terms, then $s=t$ is
% an atomic $L$-formula, and if $R\in L$ is a $k$-ary relation symbol
% and $t_1,\ldots, t_k$ are $L$-terms, then $R(t_1,\ldots, t_k)$ is an
% atomic $L$-formula. Finally, The set of $L$-formulas is the smallest
% set such that all atomic $L$-formulas are $L$-formulas, if $\phi$ is
% an $L$-formula, then $(\neg\phi)$ is an $L$-formula, if $\phi$ and
% $\psi$ are $L$-formulas, then $(\phi\wedge\psi)$ and
% $(\phi\vee\psi)$ are $L$-formulas, and if $\phi$ is an $L$-formula
% and $x\in \textsc{Var}$, then $\exists x\; (\phi)$ and
% $\forall x\;(\phi)$ are $L$-formulas.
%
We use the standard abbreviations $\phi\rightarrow \psi$ for
$\neg \phi\vee \psi$, $\phi\leftrightarrow \psi$ for
$(\phi\rightarrow \psi)\wedge (\psi\rightarrow \phi)$, and
$\phi\oplus\psi$ for
$(\phi\wedge \neg \psi) \vee (\neg \phi\wedge \psi)$.
% To improve readability we might omit parentheses with an implicit
% understanding of the binding strength of logical symbols:
% $\neg, \exists$ and $\forall$ bind more strongly than $\wedge$ which
% in turn binds more strongly than $\vee$. Finally, $\rightarrow$ and
% $\leftrightarrow$ have the least binding strength.
A \emph{theory} is a set of sentences (formulas without free
variables). A {\em model} of a theory~$T$ is a structure $\strM$ that satisfies all the sentences in $T$. 
We do not distinguish between a theory and the class of all its models.

We write $\phi(x_1,\ldots, x_k)$ to indicate that the free variables
of $\phi$ are among $x_1,\ldots, x_k$. We usually write $\bar x$ for
the tuple $(x_1,\ldots, x_k)$ and leave it to the context to determine
the length $k$ of the tuple.  Let $\strM$ be an $L$-structure. Every
$L$-formula $\phi(\bar x)$ defines a relation
$\phi(\strM)=\{\bar a\in M^{|\bar x|}~:~\strM\models \phi(\bar a)\}$.
A relation~$R$ on $\strM$ is called {\em definable} (without parameters) if there is a formula
$\phi(\bar x)$ such that $R=\phi(\strM)$. A graph $G$ is {\em definable} in $\strM$ if, for some integer $k$, we have $V(G)=M^k$ and $E(G)\subseteq M^k\times M^k$ is definable on $\strM$. A class $\Cc$ of graphs is 
\emph{definable} in a class $\Dd$ of structures if there is an integer $k$ and 
a formula $\phi(\bar x)$ with $2k$ free variables such that each 
$G\in \Cc$ is defined by $\phi$ in some $\strM\in \Dd$.

\emph{Interpretations} and \emph{transductions} provide a very useful
formalism to encode (classes of) structures inside other (classes of)
structures and to lift results from one structure to the other.  For
our purpose it will be sufficient to define interpretations of graphs
in structures. 
%An \emph{interpretation $\mathsf I$ of graphs in
%  $L$-structures} consists of two $L$-formulas $\nu(\bar x)$ and
%$\eta(\bar x,\bar y)$ with $|\bar x|=|\bar y|$. We assume that $\eta$
%is symmetric and irreflexive, that is, for all $L$-structures $\strM$
%and all $\bar a,\bar b\in M^{|\bar x|}$ we have
%$\strM\models \eta(\bar a,\bar b)\Leftrightarrow\strM\models\eta(\bar
%b,\bar a)$
%and $\strM\not\models\phi(\bar a,\bar a)$.  Note that we can
%syntactically ensure this property.  If~$\strM$ is an $L$-structure,
%then $\mathsf{I}(\strM)$ is the graph with vertex set $\nu(\strM)$ and
%edge set~$\eta(\strM)\mathop{\cap }\nu(\strM)^2$.  A \emph{simple
%  interpretation} is an interpretation with $|\bar x|=1$.
A \emph{simple interpretation $\mathsf I$ of graphs in
  $L$-structures} consists of two $L$-formulas $\nu(x)$ and
$\eta(x,y)$, where $\eta$
is symmetric (i.e.\ $\eta(x,y)\leftrightarrow\eta(y,x)$) and irreflexive (i.e.\ $\neg\eta(x,x)$).   If~$\strM$ is an $L$-structure,
then $\mathsf{I}(\strM)$ is the graph with vertex set $\nu(\strM)$ and
edge set~$\eta(\strM)\mathop{\cap }\nu(\strM)^2$. For sake of simplicity, we will also allow to define simple interpretations by means of a non-symmetric and/or reflexive formula $\psi(x,y)$ by implicitly considering the formula $\neg (x=y)\wedge(\psi(x,y)\vee \psi(y,x))$.

Transductions allow an additional coloring of the elements of
the structures, which gives additional encoding power. This is formalized
as follows. 
 For languages~$L,L^+$, if
$L\subseteq L^+$, then $L^+$ is called an \emph{expansion} 
of~$L$ and~$L$ is
called a \emph{reduct} of~$L^+$. If~$\strM^+$ is an $L^+$-structure, 
then the
structure $\strM$ obtained from $\strM^+$ by ``forgetting'' the
relations in $L^+\setminus L$ is called the $L$-reduct of $\strM^+$,
and~$\strM^+$ is called an $L^+$-expansion of $\strM$. If $L^+$ is an
expansion of~$L$ by a set of unary relation symbols, and $\strM$ is an
$L$-structure, then we call any \mbox{$L^+$-expansion $\strM^+$} of $\strM$ a
\emph{monadic lift} of $\strM$.  A \emph{transduction} $\mathsf T$ is
the composition of a monadic lift followed by a simple interpretation
$\mathsf I$. 
Let $\Cc$ and $\Dd$ be classes of $L$-structures and graphs,
respectively. We say that~$\Dd$ is a \emph{transduction} of $\Cc$ if
there exists a simple interpretation $\mathsf{I}$ of graphs in
$L^+$-structures, where $L^+$ is a monadic expansion of~$L$, such that
for every $G\in\Dd$ there exists a lift~$\strM^+$ of some structure
$\strM\in\Cc$ such that $G=\mathsf{I}(\strM^+)$.
%They play a key role in the definition of
%classes of structurally sparse graphs. 

%% file: graphs.tex
\section{Graph theoretic background}
\label{sec:graph}
We consider finite, simple and undirected graphs.  For a graph $G$ we
write~$V(G)$ for its vertex set and $E(G)$ for its edge set. For
$A,B\subseteq V(G)$ we write $E(A,B)$ for the set of edges with one
end in $A$ and one end in $B$. A partition of $V(G)$ is a family of
pairwise disjoint subsets $V_1,\ldots, V_k\subseteq V(G)$ whose union
is $V(G)$.  A bipartite graph is a graph with a vertex partition
$V_1,V_2$ such that there are no edges with both ends in $V_i$,
$i=1,2$. A graph $H$ is a subgraph of $G$ if $V(H)\subseteq V(G)$ and
$E(H)\subseteq E(G)$. For $X\subseteq V(G)$, we write $G[X]$ for the
subgraph of $G$ induced by $X$, that is, the subgraph with vertex
set~$X$ and all edges with both ends in~$X$. The graph $H$ with vertex
set~$X$ is an induced subgraph of $G$ if $H=G[X]$.  For an infinite
graph $\mathbf G$, we call the class ${\rm Age}(\mathbf G)$ of all
finite induced subgraphs of~$G$ the \emph{age} of $\mathrm{G}$.  For
disjoint subsets $X,Y$ of $V(G)$, we write $G[X,Y]$ for the subgraph
of~$G$ {\em semi-induced} by~$X$ and $Y$, that is, the subgraph with
vertex set $X\cup Y$ and all the edges with one endpoint in~$X$ and
one endpoint in $Y$. A bipartite graph~$H$ is a {\em semi-induced
  subgraph} of $G$ if $H=G[X,Y]$ for some disjoint subsets~$X$ and~$Y$
of $V(G)$.  A class $\Cc$ of graphs is called {\em monotone} if it is closed
under taking subgraphs and {\em hereditary} if it is closed under taking
induced subgraphs. For a graph~$H$, a class $\Cc$ is called {\em $H$-free}
if no $G\in \Cc$ contains $H$ as an induced subgraph.  A set
$X\subseteq V(G)$ is called {\em homogeneous} if either all distinct vertices
of~$X$ are adjacent (induce a clique) or non-adjacent (induce an
independent set). More generally, a pair $(A,B)$ of subsets of vertices is {\em homogeneous} if $G[A,B]$ is either complete bipartite  or edgeless. Note that a subset $A$ of vertices is homogeneous exactly if the pair $(A,A)$ is homogeneous.
%For a set~$C$, a coloring of $G$ with colors in~$C$
%is a mapping $c\colon V(G)\rightarrow C$. The coloring is proper if
%$\{u,v\}\in E(G)$ implies $c(u)\neq c(v)$.
A graph $G$ is {\em $d$-degenerate} if every non-empty induced subgraph of $G$ has minimum degree at most $d$. A class $\mathscr C$ is {\em degenerate} if there is an integer $d$ such that all the graphs in $\mathscr C$ are $d$-degenerate. For a graph~$G$ we denote by $\bar{\rm d}(G)$ the {\em average degree} of $G$, that is the average of the degrees of the vertices of $G$.

%% file: sparsity.tex
\medskip
\subsection{Sparse graph classes}
We refer the reader to \cite{nevsetvril2012sparsity} for an in-depth
study of the notions outlined here. The {\em $r$-subdivision} of a
graph~$G$ is the graph $G^{(r)}$ obtained by subdividing each edge 
of~$G$ exactly $r$ times. A \emph{$\leq$\,$r$-subdivision} of~$G$ is a
graph obtained by subdividing each edge of $G$ at most $r$ times. A
graph~$H$ is a {\em topological minor} of a graph $G$ {\em at depth}
$r$ if a $\leq$\,$2r$-subdivision of $H$ is a subgraph of $G$. We
denote by $G\shtm r$ the set of all the topological minors of $G$ at
depth~$r$. The two key notions in the theory of sparsity~\cite{nevsetvril2012sparsity} are the notions of \emph{bounded expansion}
and \emph{nowhere denseness}.

\smallskip
\begin{definition} 
  A class $\mathscr C$ of graphs has {\em bounded expansion} if
  there exists a function \mbox{$f\colon\mathbb N\rightarrow\mathbb N$} with
  \begin{equation} \forall G\in\mathscr C.\ \forall H\in G\shtm
    r.\quad\bar{\rm d}(H)\leq f(r).
  \end{equation} 
\end{definition}

\smallskip
\begin{definition} 
  A class $\mathscr C$ of graphs is {\em nowhere dense} if there exists a
  function $f\colon\mathbb N\rightarrow\mathbb N$ with
  \begin{equation} \label{eq:nd} \forall G\in\mathscr C.\ \forall H\in
    G\shtm r.\quad\omega(H)\leq f(r). 
  \end{equation} 
\end{definition}

\medskip
Note that every class with bounded expansion is nowhere dense.

\medskip
Bounded expansion and nowhere dense classes enjoy numerous characterizations and applications (see \cite{nevsetvril2012sparsity}). 
In fact, most graph invariants ($\alpha,\chi,\chi_f,\omega$, etc.) lead to characterizations of these classes
\cite{chi_f,nevsetvril2008grad, SurveyND}.
It also appears that for monotone classes of graphs these definitions provide a natural link to model theory (see e.g.~\cite{Adler2013,pilipczuk2018number,podewski1978stable}).
The notions of stability, monadic 
stability, dependence and monadic dependence mentioned in the next theorem are fundamental 
notions from model theory, which will be formally recalled later in \Cref{def:stability} and \Cref{def:dependent}.

\pagebreak
\begin{theorem}[Podewski, Ziegler~\cite{podewski1978stable}, Adler,
  Adler~\cite{Adler2013}]
\label{thm:adler}
If a class $\Cc$ of graphs is monotone, then the following are equivalent.
\begin{enumerate}[(i)]
\item $\Cc$ is nowhere dense, \item $\Cc$ is stable, \item $\Cc$ is
  monadically stable, \item $\Cc$ is dependent, \item $\Cc$ is
  monadically dependent.
  % \item $\Cc$ is included in the age of a superflat graph.
\end{enumerate} 
\end{theorem}

\begin{corollary}\label{crl:nd-stable}
  Every nowhere dense class $\Cc$ is monadically stable.
\end{corollary}
\begin{proof}
  According to \eqref{eq:nd} the monotone closure $\overline\Cc$ of a
  nowhere dense class $\Cc$ is nowhere dense. Hence $\overline\Cc$ is
  monadically stable (by \Cref{thm:adler}), and so is
  $\Cc\subseteq \overline\Cc$.
\end{proof}

\subsection{Rankwidth and linear rankwidth}
The notion of~rankwidth was introduced
in~\cite{oum2006approximating} as an efficient approximation to
cliquewidth.  For a graph $G$ and a subset $X\subseteq V(G)$ we define
the \emph{cut-rank} of $X$ in $G$, denoted~$\rho_G(X)$, as the rank of
the $|X|\times |V(G)\setminus X|$  matrix $A_X$ over the binary
field~$\mathbb{F}_2$, where the entry of $A_X$ on the $i$-th row and
$j$-th column is~$1$ if and only if the $i$-th vertex in~$X$ is
adjacent to the $j$-th vertex in $V(G)\setminus X$. If $X=\emptyset$
or $X=V(G)$, then we define~$\rho_G(X)$ to be zero.

A \emph{subcubic} tree is a tree where every node has degree $1$ or
$3$. A \emph{rank decomposition} of a graph~$G$ is a pair $(T,L)$,
where $T$ is a subcubic tree with at least two nodes and $L$ is a
bijection from $V(G)$ to the set of leaves of $T$.  For an edge
$e\in E(T)$, the connected components of $T-e$ induce a partition
$(X,Y)$ of the set of leaves of $T$. The \emph{width} of an edge~$e$
of $(T,L)$ is the cut-rank $\rho_G(L^{-1}(X))$. The width of $(T,L)$ is the maximum
width over all edges of $T$ (and at least $0$). The \emph{rankwidth}
$\rw(G)$ of $G$ is the minimum width over all rank decompositions of~$G$.  When the graph has at most one vertex then there is no rank
decomposition and the rankwidth is defined to be $0$.

The linear rankwidth of a graph is a linearized variant of
rankwidth, similarly as pathwidth is a linearized variant of
treewidth:  Let $G$ be an $n$-vertex graph and let $v_1,\ldots, v_n$
be an order of $V(G)$. The \emph{width} of this order is
$\max_{1\leq i\leq n-1}\rho_G(\{v_1,\ldots, v_i\})$.  The \emph{linear
  rankwidth} of $G$, denoted $\lrw(G)$, is the minimum width over all
linear orders of $G$. If $G$ has less than $2$ vertices we define the
linear rankwidth of $G$ to be zero.  An alternative way to define the
linear rankwidth is to define a linear rank decom\-posi\-tion $(T,L)$
to be a rank decomposition such that $T$ is a caterpillar and then
define linear rankwidth as the minimum width over all linear rank
decompositions. Recall that a caterpillar is a tree in which all the
vertices are within distance $1$ of a central path.

Chudnovsky and Oum \cite{chudnovsky2018vertex} observed that classes
of graphs with rankwidth at most~$k$ have the strong Erd\H os-Hajnal
property: Indeed, an $n$-vertex graph $G$ of rankwidth at most~$k$ has
a vertex set $X$ such that the cut-rank of $X$ is at most~$k$ and 
$|X|,|V(G)|-|X|>n/3$. Then one can partition each of $X$ and $V(G)-X$
into at most $2^k$ subsets such that each part of $X$ is complete or
anti-complete to each part of $V(G)-X$. It is thus natural to ask
whether classes with bounded rankwidth are distal-defined. 
We leave this question as a problem (Problem~\ref{pb:rw}).
%We answer
%this question in \Cref{sec:rw},  \Cref{lem:rw}.

\subsection{Low complexity classes}
Structurally sparse classes are classes that are transductions of
sparse classes, or, in other words, classes that can be encoded in a
sparse class by means of a coloring and a simple first-order
interpretation~\mbox{\cite{ SBE_drops, SurveyND, msrw}} (see \Cref{fig:lowcplx}).

\begin{figure}[ht]
	\begin{center}
		\includegraphics[width=\textwidth]{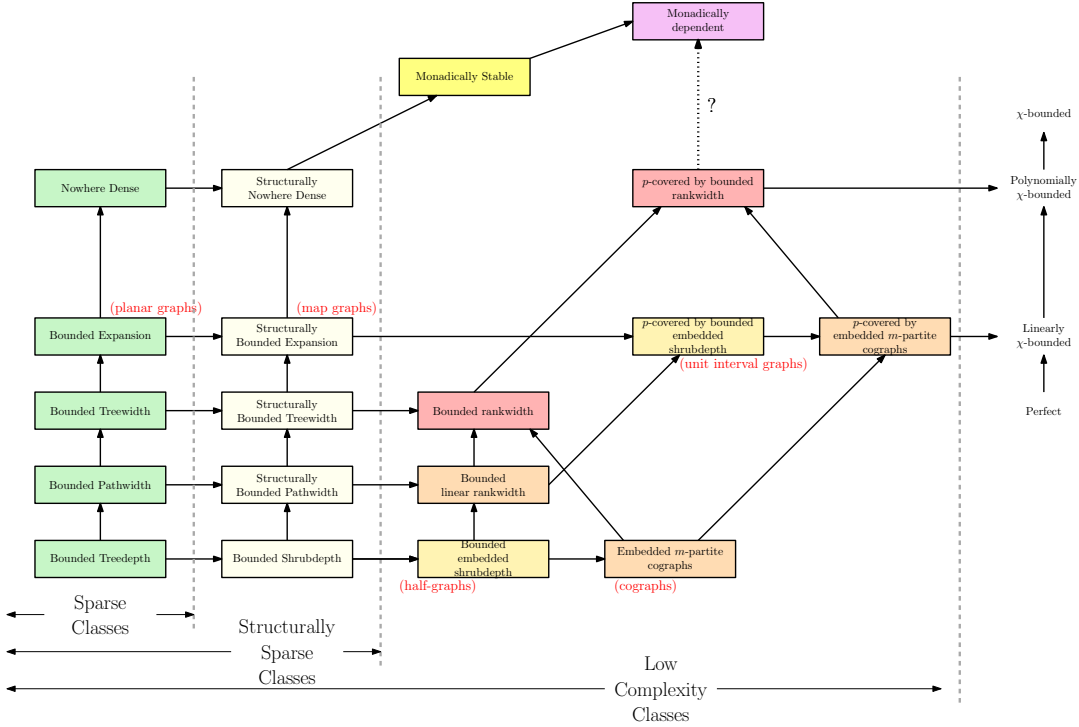}
	\end{center}
	\caption{Low complexity classes}
	\label{fig:lowcplx}
\end{figure}

For instance, classes with {\em bounded shrubdepth}~\cite{Ganian2012,
  Ganian2017} are the transductions of classes of bounded height
trees, structurally bounded expansions classes are transductions of
classes with bounded expansion, structurally nowhere dense classes are
transductions of nowhere dense classes.

\medskip
Alternatively, classes with bounded shrubdepth can be defined using a
graph invariant called {\em SC-depth}, which is inductively defined as
follows: the class $\mathcal{SC}_1$ of all graphs of SC-depth $1$ is
$\{K_1\}$ and for $t>1$, the class $\mathcal{SC}_t$ of all graphs of
SC-depth at most $t$ is the class of all graphs $G$ such that there
exists graphs $G_1,\dots,G_k$ in $\mathcal{SC}_{t-1}$ (with disjoint
vertex sets) and $A\subseteq V(G_1)\cup\dots\cup V(G_k)$, such that
$G$ is obtained from the disjoint union of $G_1,\dots,G_k$ by
complementing the adjacency of the pairs of vertices in $A\times A$.
Then a class $\mathscr C$ has bounded shrubdepth if and only if it is
included in some class $\mathcal{SC}_t$, i.e.\ if it has bounded
SC-depth.

\medskip
Graphs with SC-depth $t$ are special instances of $m$-partite cographs
(for $m=2^t$). An {\em $m$-partite cograph} is a graph that can be
encoded in a tree semilattice $(T,\wedge)$, that is, the
meet-semilattice defined by the least common ancestor
operation~$\wedge$ in the rooted tree $T$, as follows: the leaves
of $T$ (i.e.\ the maximal elements of~$T$) are the vertices of $G$ and
are colored by $c\colon V(G)\rightarrow [m]$, where $[m]=\{1,\ldots,m\}$, while each internal
vertex~$v$ of $T$ (i.e.\ each non-maximal element $v$ of~$T$) is
assigned a symmetric function $f_v\colon [m]\times [m]\rightarrow\{0,1\}$ in
such a way that two vertices $u,v$ of $G$ are adjacent if and only if
$f_{u\wedge v}(c(u),c(v))=1$. Hence, cographs are exactly $1$-partite
cographs (only one color of vertices). Note that $m$-partite cographs are clearly transductions of
{\em tree-orders}, that is of partial orders defined by the ancestor
relation in a rooted tree.  However, not every transduction of the
class of tree-orders is a class of $m$-partite cographs for some
$m$. As proved by Colcombet \cite{Colcombet2007}, a class is a
transduction of a class of tree-orders if and only if it has bounded
rankwidth, and it is a transduction of a class of linear orders
if and only if it has bounded linear rankwidth.
	
\begin{definition}
A class $\mathscr C$ is {\em $p$-covered} by a class $\mathscr D$ if
there exists an integer \mbox{$K(p)\geq p$} such that every
$G\in\mathscr C$ has a vertex partition $V_1,\dots,V_{K(p)}$ with
\mbox{$G[V_{i_1}\cup\dots\cup V_{i_p}]\in\mathscr D$} for all
$1\leq i_1<i_2<\dots<i_p\leq K(p)$. The minimum integer $K(p)$ is the
{\em magnitude} of the $p$-cover. If a class $\mathscr C$ is $p$-covered by a class~$\mathscr D$ for each integer $p$, we say that 
$\mathscr C$ has {\em low $\mathscr D$-covers}.
\end{definition}

We have the following non-trivial characterizations of classes with bounded
expansion and of classes with structurally bounded expansion.

\begin{theorem}[\cite{nevsetvril2008grad}]
  A class $\mathscr C$ has bounded expansion if and only if for every
  integer $p$ the class $\mathscr C$ is $p$-covered by a class with
  bounded treedepth.
\end{theorem}

\begin{theorem}[\cite{SBE_drops}]\label{thm:sbe}
  A class $\mathscr C$ has structurally bounded expansion if and only
  if for every integer $p$ the class $\mathscr C$ is $p$-covered by a
  class with bounded shrubdepth.
\end{theorem}

Classes with low bounded rankwidth covers have been
considered in~\cite{kwon2020low} and 
with low bounded linear rankwidth covers are discussed in~\cite{SODA_msrw}.
In particular, it is proved in \cite{kwon2020low} that interval graphs
and permutations graphs are not $3$-covered by any class with bounded
rankwidth.

%% file: regularity.tex
\section{Random-free regularity lemmas}

We first survey regularity properties of hereditary graph classes that
are defined by excluding semi-induced bipartite graphs. 
Common feature of these results is that the regular pairs fail to be random-like bipartite graphs.
 The different
types of regularity considered here are summarized in \Cref{tab:srls}.

\subsection{VC-dimension}
We start with graph classes of bounded VC-dimension. A hereditary
class of graphs has bounded VC-dimension if and only if it excludes
some bipartite graph as an induced subgraph.  Usually, VC-dimension is
defined for set families~\cite{Vapnik1971}, however, in the context of
graph theory the following equivalent definition is more convenient.

\begin{definition}
  The VC-dimension of a graph $G$ is the largest integer $d$ such that
  there exist vertices $a_1,\ldots, a_d\in V(G)$ and vertices
  $b_J\in V(G)$ for $J\subseteq \{1,\ldots,d\}$ such that
  $\{a_i,b_J\}\in E(G) \Leftrightarrow i\in J$.
\end{definition}

Note that a hereditary class $\mathscr C$ has bounded VC-dimension if
and only if the number of graphs in $\mathscr C$ on $n$ vertices is at
most $2^{n^{2-\epsilon}}$ for some $\epsilon>0$, as proved by Alon,
Balogh, Bollob\'as, and Morris \cite{alon2011structure}.

\begin{table}[ht]
  \begin{tabular}{|c@{$\qquad$}c|}
    \hlx{ssshvv}
    \multicolumn{2}{|c|}{\Large Variants of Szemer\'edi's Regularity Lemma}\\
    \hlx{vvhv}
    \begin{minipage}{.5\textwidth}
      General case: all pairs but an \mbox{$\epsilon$-fraction} are
      \mbox{{\em $\epsilon$-regular}}. This means that for most pairs
      $(A,B)$ of parts, the density of edges between subsets of $A$
      and $B$ (with at least~$\epsilon$ relative size) differ from the
      density of edges between $A$ and $B$ by at most~$\epsilon$.
    \end{minipage}
 &
   \begin{minipage}{.3\textwidth}
     \includegraphics[width=\textwidth]{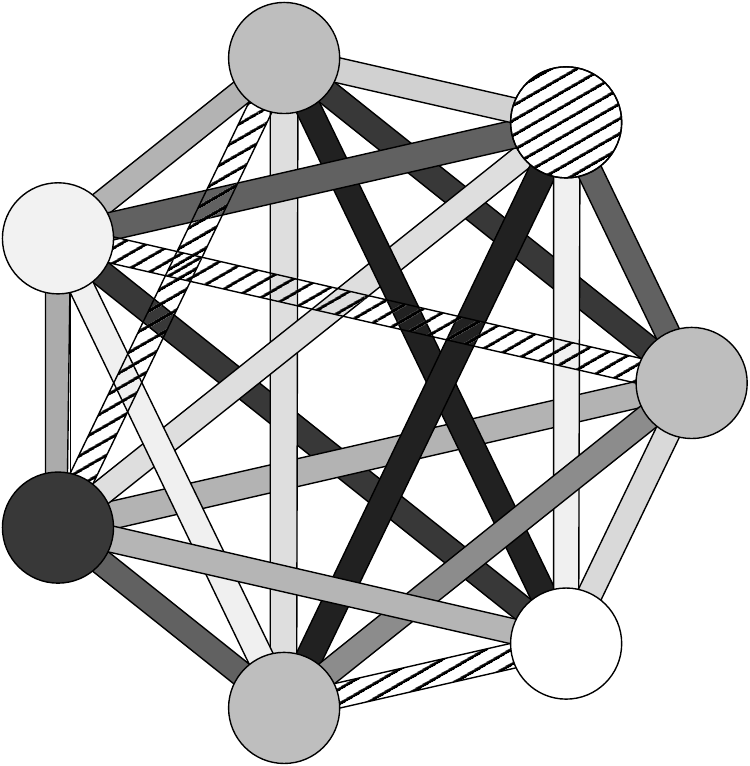}
   \end{minipage}\\
    \hlx{vhv}
    \begin{minipage}{.5\textwidth}
      Bounded VC-dimension, NIP: all pairs but an $\epsilon$-fraction
      are \mbox{{\em $\epsilon$-homogeneous}}. This means that their
      densities are either $<\epsilon$ or $>1-\epsilon$.
 	
      The number of parts is polynomial.
    \end{minipage}
 &
   \begin{minipage}{.3\textwidth}
     \includegraphics[width=\textwidth]{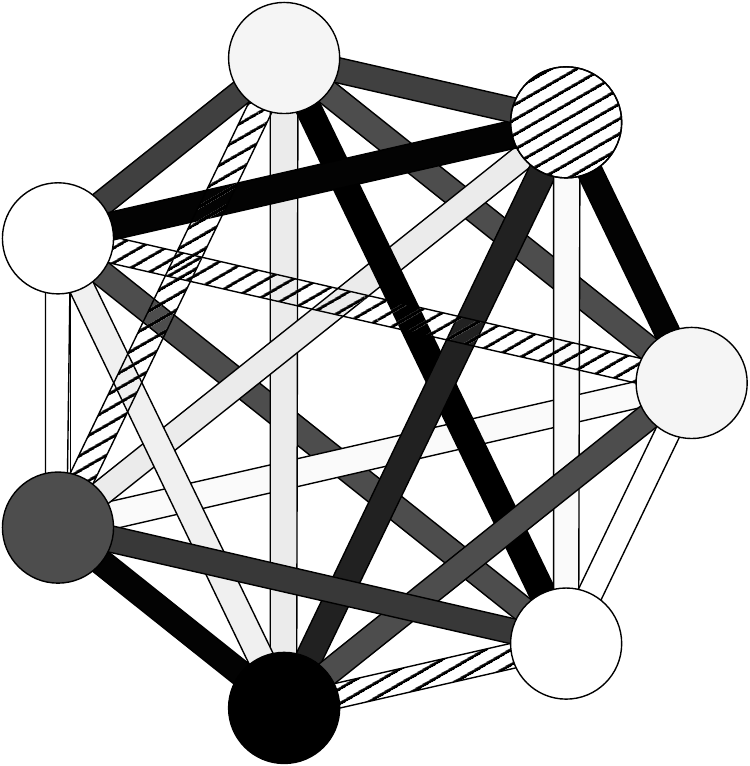}
   \end{minipage}\\
    \hlx{vhv}
    \begin{minipage}{.5\textwidth}
      Bounded order-dimension, stable: all parts are {\em
        $\epsilon$-excellent} and all the pairs are {\em
        $\epsilon$-uniform}. In particular, every part~$A$, all the
      vertices of $G$ have degree either $<\epsilon |A|$ in~$A$ or at
      least $>(1-\epsilon)|A|$ in~$A$, and that for every pair $(A,B)$
      of parts at least $(1-\epsilon)$ proportion of the vertices in
      $A$ have similar degree in $B$.
  
      The number of parts is polynomial.
    \end{minipage}
 &
   \begin{minipage}{.3\textwidth}
     \includegraphics[width=\textwidth]{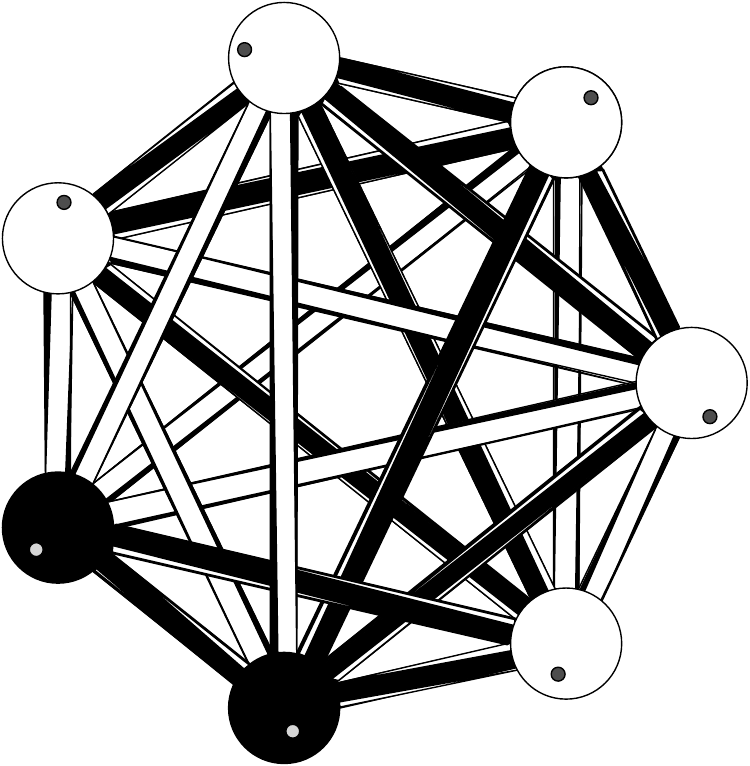}
   \end{minipage}\\
    \hlx{vhv}
    \begin{minipage}{.5\textwidth}
      Induced subgraph of a graph definable in a distal structure,
      semi-algebraic: all pairs but an $\epsilon$-fraction are {\em
        homogeneous}. This means that between non exceptional pairs
      either all edges are present or no edge is present.
  
      The number of parts is polynomial.
    \end{minipage}
 &
   \begin{minipage}{.3\textwidth}
     \includegraphics[width=\textwidth]{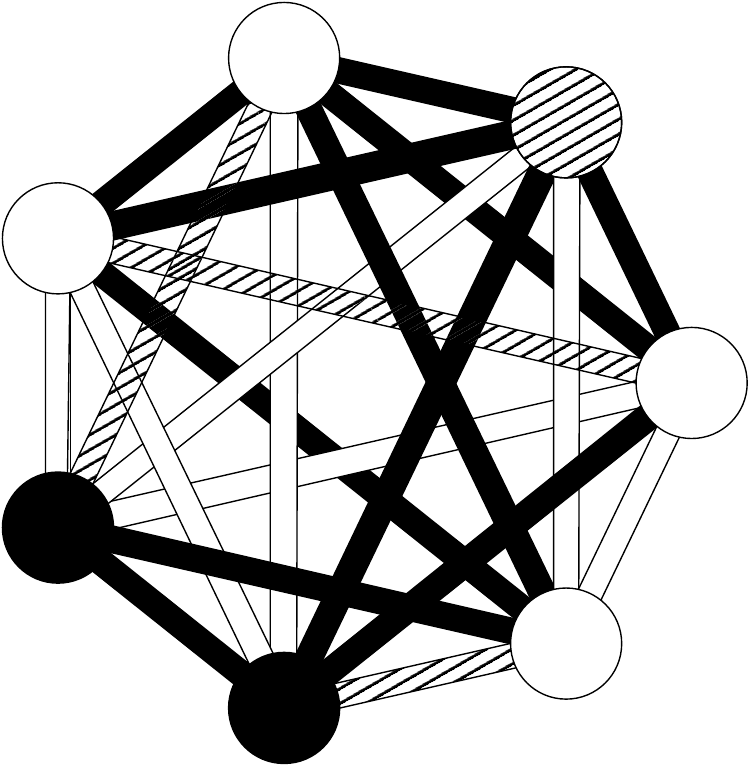}
   \end{minipage}\\
    \hlx{vhs[10pt]}
  \end{tabular}
  \caption{Summary of the different types of regularity lemmas considered in this section. Hatched zones correspond to irregular parts or pairs.\label{tab:srls}}
\end{table}

\pagebreak
The notion of VC-dimension is strongly related to the model theoretic
notion of \emph{dependence} (or \emph{NIP}) ~\cite{Shelah2004,simon2015guide}. 

\begin{definition}\label{def:dependent}
A class $\Cc$ of structures is \emph{dependent} if every graph class
definable in $\Cc$ has bounded VC-dimension. 

The class~$\Cc$ is \emph{monadically dependent} if every graph class
definable in the class $\{\strM^+\mid \strM^+$ monadic lift of 
$\strM\in \Cc\}$ of all monadic lifts of structures from $\Cc$ has
bounded VC-dimension. 
\end{definition}

%Let $\phi(\bar x,\bar y)$ be a formula. For 
%every structure $\strM$ the formula~$\phi$ defines a graph
%$G_{\strM,\phi}$ with vertex set $M^{|\bar x|}\mathop{\dot\cup} M^{|\bar y|}$
%(if $|\bar x|=|\bar y|$ we take two disjoint copies of $M^{|\bar x|}$)
%and two elements $\bar a\in M^{|\bar x|}$ and $\bar b\in M^{|\bar y|}$ 
%are connected by an edge if $\strM\models\phi(\bar a,\bar b)$. 
%
%A formula $\phi(\bar x,\bar y)$ is called \emph{dependent} over 
%a class $\Cc$ of structures, or
%\emph{has the non-independence property NIP}, if the class 
%$\{G_{\strM,\phi}~:~\strM\in \Cc\}$ has bounded VC-dimension. 
%A class $\Cc$ of $L$-structures is \emph{dependent} if every 
%formula $\phi(\bar x,\bar y)$ is dependent over $\Cc$. 
%A class $\Cc$ of $L$-structures is \emph{monadically dependent}
%if the class $\{\strM^+~:~\strM^+$ monadic lift of $\strM\in \Cc\}$
%is dependent. 

\begin{theorem}[\cite{BS1985monadic}, see also \cite{anderson1990tree}]
A class $\mathscr C$ is monadically dependent if and only if every transduction of $\mathscr C$ is dependent.
\end{theorem}

%By definition, a class $\Cc$ of graphs has bounded VC-dimension if and only 
%if the formula $E(x,y)$ is dependent over $\Cc$. 

\begin{definition}
  Let $\epsilon>0$, let $G$ be a graph and let $A,B\subseteq V(G)$ be
  two disjoint non-empty subsets of vertices. The pair $(A,B)$ is
  called \emph{$\epsilon$-homogeneous} if $\dens(A,B)<\epsilon$ or
  $\dens(A,B)>1-\epsilon$.
\end{definition}

\begin{remark}
  Assume $(A,B)$ is $\epsilon^3$-homogeneous, let $A'\subseteq A$ and
  $B'\subseteq B$ with $|A'|\geq \epsilon |A|$ and
  $|B'|\geq \epsilon |B|$. Then, if ${\rm dens}(A,B)<\epsilon^3$ we
  have
  \[ {\rm dens}(A',B')=\frac{|E(A',B')|}{|A'|\,|B'|}\leq
  \frac{|E(A,B)|}{\epsilon^2 |A|\,|B|}<\epsilon.
  \]
  Similarly, considering the complement graph, if
  ${\rm dens}(A,B)>1-\epsilon^3$ we have
  ${\rm dens}(A',B')>1-\epsilon$.  We deduce that an
  $\epsilon^3$-homogeneous pair is $\epsilon$-regular.
\end{remark}

The following theorem is also known as the \emph{ultra-strong regularity
  lemma} for graphs with bounded VC-dimension~\cite{alon2007efficient,
  Lovasz2010,fox2019erdHos}. The presented bounds come
from~\cite{fox2019erdHos}.

\begin{theorem}[Bounded VC-dimension regularity lemma; Fox, Pach and Suk~\cite{fox2019erdHos}]
  Let $0<\epsilon<1/4$ and let $G$ be a graph of VC-dimension at most
  $d$. Then there exists an equipartition of $V(G)$ into $k$ classes
  $V_1,\ldots, V_k$, where $k$ satisfies  $8/\epsilon\leq k\leq c\cdot (1/\epsilon)^{2d+1}$
  for some constant $c$ depending only on~$d$, such that all but at
  most $\epsilon k^2$ of the pairs $(V_i,V_j)$ are
  $\epsilon$-homogeneous.
\end{theorem}

 Erd\H os and Hajnal~\cite{erdos1989ramsey} proved that for
every proper hereditary graph class~$\Cc$ there exists a constant $c$
such that every \mbox{$n$-vertex} graph $G\in \Cc$ contains a
homogeneous set of size $e^{c\sqrt{\log n}}$. They conjectured that for
every proper hereditary graph class $\Cc$ there exists a
constant~$\delta$ such that every $n$-vertex graph $G\in \Cc$ must
contain a homogeneous set of size~$n^{\delta}$.  A graph class with
this property is said to have the \emph{Erd\H os-Hajnal property}. Fox, Pach and Suk~\cite{fox2019erdHos} also proved that graphs of
bounded VC-dimension \emph{almost} have the Erd{\H o}s-Hajnal
property:

\begin{theorem}[Fox, Pach and Suk~\cite{fox2019erdHos}]
  Every $n$-vertex graph with bounded VC-dimension contains a
  homogeneous set of size at least $e^{(\log n)^{1-o(1)}}$.
\end{theorem}

Following \cite{fox2008erdHos}, we say that 
a graph class $\Cc$ has the \emph{strong Erd\H os-Hajnal property}  if
there exists a constant $\delta>0$ such that every $n$-vertex graph
$G\in\Cc$ contains a homogeneous pair $(A,B)$, where $A$ and $B$ are  two disjoint sets  of
 at least $\delta n$ vertices each.

\subsection{Order-dimension}
The study of structures without {\em $k$-order property} or, equivalently with bounded order dimension, has been initiated by Shelah in his study of stability \cite{shelah1969stable}.
%We now consider graph classes of bounded
%order-dimen\-sion (or {\em ladder dimension}), 
Order-dimension is also related to
Littlestone-dimension, which is a combinatorial parameter that
characterizes error bounds in online learning (see
\cite{chase2019model}).

\begin{definition}
  The \emph{order-dimension} of a graph $G$ is the largest integer
  $\ell$ such that there exist vertices
  $a_1,\ldots, a_\ell, b_1,\ldots, b_\ell\in V(G)$ with
  $\{a_i,b_j\}\in E(G)\Leftrightarrow i\leq j$.
\end{definition}

The notion of order-dimension is strongly related to the model theoretic
notion of \emph{stability}. 

\begin{definition}\label{def:stability}
A class $\Cc$ of structures is \emph{stable} if every graph class
definable in~$\Cc$ has bounded order-dimension. 

The class~$\Cc$ is \emph{monadically stable} if every graph class
definable in the class $\{\strM^+\mid \strM^+$ monadic lift of 
$\strM\in \Cc\}$ of all monadic lifts of structures from $\Cc$ has
bounded order-dimension. 
\end{definition}

\begin{theorem}[\cite{BS1985monadic}]
A class $\mathscr C$ is monadically stable if and only if every transduction of $\mathscr C$ is stable.
\end{theorem}

\begin{definition}
  Let $\epsilon>0$, let $G$ be a graph and let $A\subseteq V(G)$. The
  set $A$ is called \emph{$\epsilon$-good} when for every $b\in V(G)$
  either
  \[|\{a\in A \mid \{a,b\}\in E(G)|\leq \epsilon|A|\]
  or \[|\{a\in A \mid \{a,b\}\in E(G)|\geq (1-\epsilon)|A|.\]
  In the following, we write $E(a,b)$ for the boolean value
  \emph{true} if $\{a,b\}\in E(G)$ and \emph{false} otherwise. Then
  the above reads as: for every $b\in V(G)$ there exists a boolean
  value $\mathsf t(b/A)$ such that
  \[
  |\{a\in A\mid E(b,a)\neq\mathsf t(b/A)\}|\leq \epsilon|A|.
  \]

  The set $A$ is \emph{$\epsilon$-excellent} if $A$ is $\epsilon$-good
  and if $B\subseteq V(G)$ is $\epsilon$-good, then there exists a
  boolean value $\mathsf t(A/B)$ such that
  \[
  |\{a\in A\mid \mathsf t(a/B)\neq \mathsf t(A/B)\}|\leq \epsilon |A|.
  \]
  A pair $(A,B)$ satisfying this latter condition is called
  \emph{$\epsilon$-uniform}. In other words, a pair $(A,B)$ is
  $\epsilon$-uniform if all but at most $\epsilon |A|$ vertices of $A$
  have a degree in $B$ that is smaller than $\epsilon |B|$ or all but
  at most $\epsilon |A|$ vertices of $A$ have a degree in $B$ that is
  greater than $(1-\epsilon)|B|$.
\end{definition}

\begin{theorem}[Stable regularity lemma; Malliaris and Shelah~\cite{malliaris2014regularity}]
\label{thm:stable_reg}
For every~$\ell$ and every \mbox{$\epsilon>0$} there exist $M$ and $N$
such that for every graph $G$ with $n\geq N$ vertices of order-dimension
at most $\ell$, there is an equipartition of the vertex set into $k$
classes $V_1,\ldots, V_k$, $k\leq M$, where each of the pieces is
$\epsilon$-excellent, all of the pairs are $\epsilon$-uniform and if
$\epsilon<1/2^{2^\ell}$, then
$M(\epsilon,\ell)\leq
(3+\epsilon)\left(\frac{8}{\epsilon}\right)^{2^\ell}$.
\end{theorem}

Malliaris and Shelah also showed that graphs of bounded
order-dimension have the Erd\H os-Hajnal property.

\begin{theorem}[\cite{malliaris2014regularity, chernikov2016definable}]
  For every integer $\ell$ there is a constant $\delta > 0$ such that
  every $n$-vertex graph $G$ of order-dimension at most $\ell$
  contains a homogeneous subset of size at least $n^\delta$.
\end{theorem}

\subsection{Weakly-sparse classes}

Forbidding a biclique (i.e.\ a complete bipartite graph $K_{s,s}$) as
a semi-induced subgraph is equivalent, by a standard Ramsey argument,
to forbidding a clique and an induced biclique, which is in turn
equivalent to excluding some biclique as a (non induced) subgraph.  A
monotone class has bounded VC-dimension if and only if it excludes a
biclique.  Excluding a biclique as a subgraph implies strong
properties (see for instance \cite{dvovrak2018induced,
  kuhn2004induced, msrw}).  A class $\mathscr C$ that excludes a
biclique as a subgraph is called {\em weakly sparse} \cite{msrw}.

%As we shall see now, the stable regularity lemma is in some sense
%trivial for weakly sparse classes.

\begin{observation}
  For all integers $s$ and $k$ and every $\epsilon>0$ there exists an
  integer~$n$ such that every $K_{s,s}$-free graph $G$ of order at least $n$
  has the property that every equipartition of $V(G)$ in $k$ parts is
  $\epsilon$-uniform.
\end{observation}
\begin{proof}
Fix integers $s,k$ and $\epsilon>0$.
  According to \cite{kovari1954problem}, there exist a constant~$C$
  such that ${\rm ex}(n,K_{s,s})\leq C n^{2-\frac{1}{s}}$.  Thus if $n$ is
  sufficiently large, the average degree~$\bar{\rm d}(G)$ of $G$
  is at most $\epsilon^2 n/k^2$.  It follows that $G$ contains at most
  $\epsilon n/k$ vertices of degree greater than $\epsilon n/k$. It
  follows that every equipartition of $V(G)$ into $k$ parts is
  $\epsilon$-uniform.
\end{proof}

However, it is not clear whether one can require that all the parts are $\epsilon$-excellent in some
  partition of size $(1/\epsilon)^{cs}$, for some universal constant
  $c$.

\subsection*{Classes in the age of an infinite structure}

We now consider classes included in the age of infinite
structures. Precisely, starting from some ``nice'' infinite structure
$\strM$ (like the real field $(\R,+,\cdot,<,0,1)$, the dense linear order~$(\Q,<)$, or the infinite set~$\N$) we first
construct an infinite graph $\mathbf U$ definable in $\strM$, which is
an infinite graph whose vertex set is $M^d$ for some $d$, and whose
adjacency is given by a definable relation. We then consider
classes~$\Cc$ of graphs with
\mbox{$\Cc\subseteq{\rm Age}(\mathbf U)$}. Examples of this general scheme are abundant in both model theory and combinatorics.

\subsection{Semi-algebraic and distal-defined classes}

A class $\mathscr C$ is {\em semi-algebraic} if there are polynomials 
\[f_1,\dots,f_t\in\mathbb R[x_1,\dots,x_d,y_1,\dots,y_d]\]
and a Boolean function $\Phi$ such that for every graph
$G\in\mathscr C$ there exists a mapping
$p\colon V(G)\rightarrow\mathbb R^d$ with
\[
\{u,v\}\in E(G)\iff\Phi(f_1(p(u),p(v))\geq 0; \dots;
f_t(p(u),p(v))\geq 0)=1.
\]
In other words, if we let $\mathbf U$ to be the graph with vertex set
$\mathbb R^d$ and edge set
\[
E(\mathbf U)=\{(\mathbf x,\mathbf y)\in \mathbb R^d\times\mathbb
R^d\mid \Phi(f_1(\mathbf x,\mathbf y)\geq 0; \dots; f_t(\mathbf
x,\mathbf y)\geq 0)=1\},
\]
then every graph in $\mathscr C$ is a finite induced subgraph of
$\mathbf U$. Note that real closed fields have quantifier elimination
and hence the above is equivalent to stating that the graph
$\mathbf U$ is definable in $(\R,+,\cdot,<,0,1)$.  We say that
$\mathscr C$ has complexity~$(t,D)$ if each polynomial $f_i$ with
$1\leq i\leq t$ has degree at most $D$.

\begin{example}
  Intersection graphs of segments and intersection graphs of balls
  in~$\mathbb R^d$ are examples of semi-algebraic classes.
\end{example}

Alon, Pach, Pinchasi, Radoi\v ci\'c, and
Sharir~\cite{alon2005crossing} proved that for semi-algebraic graphs
with bounded description complexity, the pairs in the regularity lemma
can be required to be homogeneous instead of $\epsilon$-regular.  This
result has been extended by Fox, Gromov, Lafforgue, Naor, and
Pach~\cite{fox2012overlap} to $k$-uniform hypergraphs and, in this
more general setting, Fox, Pach, and Suk~\cite{fox2014density} proved
that a polynomial number of parts are sufficient and that
semi-algebraic classes have the strong Erd\H os-Hajnal property.
 
 \begin{theorem}[Semi-algebraic regularity lemma; Fox, Pach, and Suk~\cite{fox2014density}]
 \label{thm:sa}
   For all integers $d,D,t\geq 1$ there exists a constant $c$ such
   that for every $0<\epsilon<1/2$ and every semi-algebraic graph $G$
   in~$\R^d$ with complexity $(t,D)$ there is an equipartition
   $V_1,\ldots, V_k$ of the vertex set into~$k$ classes with
   $k\leq (1/\epsilon)^c$ such that all but at most $\epsilon k^2$
   pairs are homogeneous.
\end{theorem}

The real field $(\mathbb R, +,\cdot,<,0,1)$ is an example of so-called
distal structures, and the above results have then been extended to
classes of graphs included in the age of a graph definable in a distal
structure \cite{chernikov2015regularity}, which we call {\em
  distal-defined} classes.
 
The notion of distal theories was defined in \cite{simon2013distal} to
isolate the class of ``purely unstable'' dependent theories. The original
definition is in terms indiscernible sequences, but a more
combinatorial characterization can be found in
\cite{chernikov2015externally}.  While stability allows a short
combinatorial definition, distality is a lot more complicated and we
refrain from giving a formal definition here. Apart from real closed
fields, an example of distal theories is the theory of dense linear
orders without endpoints, $p$-adic fields with valuation, and
Presburger arithmetic.
We now state the graph version of \cite[Theorem
5.8]{chernikov2015regularity} in our setting.

\begin{definition}\label{def:nice-decomposition}
	Let $G$ be an $n$-vertex graph and let $\epsilon>0$. A partition $V_1, \ldots,V_k$ of $V(G)$ is {\em $\epsilon$-nice} if 
	 \[
    \sum_{\text{non-homogeneous }(V_i,V_j)}\frac{|V_i|\,|V_j|}{n^2}<\epsilon.
    \]
\end{definition}
Note that if a partition $V_1,\dots,V_k$ is an equipartition, then it is $\epsilon$-nice if and only if all pairs but an $\epsilon$-fraction are homogeneous.

\begin{theorem}[Distal regularity lemma; Chernikov and Starchenko \cite{chernikov2015regularity}]
\label{thm:distal_reg}
For every distal-defined class $\mathscr C$ there is a constant $c$
such that for every $\epsilon>0$ and for every $n$-vertex graph
$G\in\mathscr C$, there exists an $\epsilon$-nice partition $V=V_1\cup\dots\cup V_k$ of
$V(G)$ with $k \leq (1/\epsilon)^c$. 
% such that
%  \[
%  \sum_{(V_i,V_j)\text{ non-homogeneous}}\frac{|V_i|\,|V_j|}{n^2}\leq
%  \epsilon.
%  \]
\end{theorem}

The distal regularity lemma stated above is similar in its form to the 
Frieze-Kannan (weak) regularity lemma \cite{frieze1996regularity}.
However, it is easy to deduce a version with an equipartition from \Cref{thm:distal_reg}, by applying the next lemma.
\begin{lemma}
\label{lem:equipartition}
	Assume $G$ has an $\epsilon$-nice partition into $k$ classes.
	Then $V(G)$ has an equipartition into $k/\epsilon$ classes such that all but (at most) a $3\epsilon$-fraction of the pairs are homogeneous.
\end{lemma}
\begin{proof}
	Let $V_1\cup\dots\cup V_k$ be an $\epsilon$-nice partition of $G$ into $k$ parts and let $\Sigma\subseteq[k]\times [k]$ be the set of all pairs $(i,j)$ such that $(V_i,V_j)$ is not homogeneous.
		Let $K=\lceil k/\epsilon\rceil$. For the sake of simplicity we assume that $K$ divides $n$.
	We split each part~$V_i$ into $W_{i,0},W_{i,1},\dots,W_{i,a_i}$ with $|W_{i,0}|\leq |W_{i,1}|=\dots=|W_{i,a_i}|=n/K$.

	Let $\mathcal I=\{(i,s)\mid 1\leq i\leq k\text{ and }1\leq s\leq a_i\}$. 
		Note that each pair $(W_{i,s},W_{j,t})$ with $(i,j)\notin\Sigma$ is homogeneous. 
		Let $\Sigma'$ be the set of the pairs $(i,s),(j,t)\in\mathcal I$ such that $(i,j)\in \Sigma$.
		Let $Z=\bigcup_i W_{i,0}$. Note that $|Z|\leq kn/K$, say, 
		$|Z|=k'n/K$ for some $k'\leq k$. 
	We now consider an equipartition of $Z$ into sets $Z_1,\dots,Z_{k'}$ of size~$n/K$.
		As $(V_1,\dots,V_k)$ is $\epsilon$-nice we have $\sum_{(i,j)\in\Sigma}|V_i|\,|V_j|< \epsilon n^2$.
	It follows that 
	\[
	\sum_{((i,s),(j,t))\in\Sigma'}|W_{i,s}|\,|W_{j,t}|<\epsilon n^2.
	\]
	As $|W_{i,s}|=n/K$ we get $|\Sigma'|<\epsilon K^2$.
	It follows that the global number of non-homogeneous pairs is 
	bounded by $|\Sigma'|+Kk'+k'^2\leq 3k^2/\epsilon$.
	Hence the proportion of non-homogeneous pairs is at most
	$(3k^2/\epsilon)/(k^2/\epsilon^2)=3\epsilon$. 
\end{proof}

%\begin{proof}
%	Let $V_1\cup\dots\cup V_k$ be an $\epsilon$-nice partition of $G$ into $k$ parts and let $\Sigma\subseteq[k]\times [k]$ be the set of all pairs $(i,j)$ such that $(V_i,V_j)$ is not homogeneous.
%		Let $K=\lceil k/\epsilon\rceil$. For the sake of simplicity we assume that $K$ divides $n$.
%	We split each part~$V_i$ into $W_{i,0},W_{i,1},\dots,W_{i,a_i}$ with $|W_{i,0}|\leq |W_{i,1}|=\dots=|W_{i,a_i}|=n/K$.
%
%	Let $\mathcal I=\{(i,s)\mid 1\leq i\leq k\text{ and }1\leq s\leq a_i\}$. Let $K_0=|\mathcal I|$. 
%		Note that each pair $(W_{i,s},W_{j,t})$ with $(i,j)\notin\Sigma$ is homogeneous. 
%		Let $\Sigma'$ be the set of the pairs $(i,s),(j,t)\in\mathcal I$ such that $(i,j)\in \Sigma$.
%		Let $Z=\bigcup_i W_{i,0}$. Note that $|Z|\leq kn/K$, hence, $K_0\geq K-k$.
%	We now consider an equipartition of $Z$ into sets $Z_1,\dots,Z_{K-K_0}$ of size $n/K$.
%		As $(V_1,\dots,V_k)$ is $\epsilon$-nice we have $\sum_{(i,j)\in\Sigma}|V_i|\,|V_j|< \epsilon n^2$.
%	It follows that 
%	\[
%	\sum_{((i,s),(j,t))\in\Sigma'}|W_{i,s}|\,|W_{j,t}|<\epsilon n^2.
%	\]
%	As $|W_{i,s}|=n/K$ we get $|\Sigma'|<\epsilon K^2$.
%	It follows that the global number of non-homogeneous pairs is bounded by $|\Sigma'|+K(K-K_0)\leq K(\epsilon K+k)$.
%	Hence the proportion of non-homogeneous pairs is at most
%	$\epsilon+k/K<2\epsilon$.
%\end{proof}

It is also proved in \cite{chernikov2015regularity} that  distal-defined classes of
graphs have the strong Erd\H os-Hajnal property.

\bigskip

%% file: set.tex
\section{Set-defined classes: both stable and semi-algebraic regularity}
\label{sec:setdef}
We call a class $\Cc$ {\em set-defined} if it is included in the age
of a graph definable in~$\mathbb N$, considered as a model of an
infinite set.  Note that every set-defined class is obviously semi-algebraic.
%Although $\mathbb N$ is not distal, set-defined classes
%are semi-algebraic and distal-defined. 
However,
as $\mathbb N$ is stable,  every set-defined class not
only has bounded VC-dimension, but also has bounded order-dimension.
It follows that set-defined classes enjoy both stable regularity (\Cref{thm:stable_reg}) and semi-algebraic regularity (\Cref{thm:sa}).
For this reason, it seems that it is worth studying these classes. Moreover, set-defined classes possess properties that make them a dense analog of degenerate classes (\Cref{thm:conj1}). 

\begin{example}
  The class of cographs is not set-defined. Indeed, the
  order-dimension of cographs is unbounded.
\end{example}

%\begin{example}
%  For fixed integer $k$, the class
%  $\{{\rm KG}_{n,k}\mid n\in\mathbb N\}$ of Kneser graphs is
%  set-defined. Indeed, each ${\rm KG}_{n,k}$ is an induced subgraph of
%  ${\rm KG}_{\omega,k}$, which is definable in $\mathbb N$.
%\end{example}
\begin{example}
	The {\em shift-graph} $S(n,k)$ has vertex set 
	$V=\{\bar x\in [n]^k\mid x_1<x_2<\dots<x_k\}$ and edge set 
	$E=\{\{\bar x,\bar y\}\mid \bigwedge_{i=1}^{k-1}(x_i=y_{i+1})\,\vee\,\bigwedge_{i=1}^{k-1}(y_i=x_{i+1})\}$. It follows that for fixed $k$ the class $\{S(n,k)\mid n\in\mathbb N\}$ is set-defined.
	Note that, however, this class is not $\chi$-bounded, as shift-graphs are triangle-free and $\chi(S(n,k))=(1+o(1))\underbrace{\log\dots\log}_{k-1\text{ times}}n$ \cite{erdos1968chromatic}.
\end{example}

\begin{lemma}\label{lem:sd-setdef}
Every graph class with bounded shrubdepth is set-defined. 
\end{lemma}
\begin{proof}
A class $\Cc$ of graphs has bounded shrubdepth if and only if 
it has bounded SC-depth. The notion of SC-depth leads to a natural
notion of \emph{SC-decompositions}. An SC-decomposition of a graph $G$
of SC-depth at most $d$ is a rooted tree $T$ of depth $d$ with leaf
set $V(G)$, equipped with unary predicates $A_1,\ldots, A_d$ on 
the leaves. Each child $s$ of the root in $T$ corresponds to one of the 
subgraphs $G_1,\ldots, G_k$ of SC-depth $d-1$, such that $G$ is obtained from the 
disjoint union of the $G_i$ by complementing the adjacency of 
the pairs of vertices in $A_1\times A_1$. We continue recursively 
with the subgraphs $G_i$ using the predicate $A_j$ at level $j$
of the tree $T$. 

We now show that for fixed $d$ the class of graphs of SC-depth $d$
is set-defined by a formula $\phi(\bar x,\bar y)$ with $|\bar x|=d+1$. 
Let \[\phi(\bar x,\bar y)\coloneqq \neg(x_{d+1}=y_{d+1})\wedge
\bigoplus_{i=1}^{d}(x_i=y_i).\]

Let $G\in \Cc$ and fix an SC-decomposition $T$ of $G$, together with some injection $l\colon V(T)\rightarrow \N$. For $v\in V(G)$ we call $l(v)$ 
the label of $v$. Now we map each vertex $v\in V(G)$
to the tuple $(x_1,\ldots, x_d, x_{d+1})$, where $x_{d+1}$ is the label
of $v$ and for $1\leq i\leq d$, $x_i$ is either the label of the ancestor
of $v$ at depth $i$ in $T$, if the vertex belongs to the complemented set
$A_i$ at level $i$, or the label of $v$ otherwise. It is easy to verify
this mapping induces an isomorphism of $G$ and its image in the infinite
graph $\mathbf{U}$ defined by $\phi$ on $\N^{d+1}$. 
\end{proof}
We will use the lemma to prove in \Cref{crl:rwcovers-setdefined} that
a class has structurally bounded expansion if and only if it has low
linear rankwidth covers and is set-defined.

Our motivation to introduce 
set-defined classes is that they have bounded order-dimension and are
semi-algebraic. This naturally leads to the following problem.

\begin{problem}
  Is there a variant of the regularity lemma for set-defined classes,
  which would imply (for set-defined classes) both the semi-algebraic version (\Cref{thm:sa}) and the
  stable version (\Cref{thm:stable_reg})?
\end{problem}

%%%%%%%%%%%%%%%%%%%

We now show that set-defined classes can be seen as a dense analog of degenerate classes.

\begin{lemma}
\label{lem:deg}
Every degenerate class is set-defined.
\end{lemma}
\begin{proof}
  Let $\mathscr C$ be a $d$-degenerate class of graphs.  We consider
  the following formula~$\phi(\bar x,\bar y)$, where $\bar x$ and
  $\bar y$ are $d+1$-tuples.
  \[
  \phi(\bar x,\bar y):=\bigvee_{i=1}^{d+1}(x_{d+1}=y_i)\vee
  (y_{d+1}=x_i).
  \]
  This defines the adjacency of an infinite graph $\mathbf U$ with
  vertex set $\mathbb N^{d+1}$.

  For every graph $G\in\mathscr C$ there is a numbering
  $\ell\colon V(G)\rightarrow \{1,\dots,|G|\}$ such that every vertex
  $v$ has at most $d$ neighbors $u$ with $\ell(u)<\ell(v)$. We define
  $f\colon V(G)\rightarrow\mathbb N^{d+1}$ as follows: for every
  vertex $v\in V(G)$ such that $u_1,\dots,u_k$ are the neighbors of
  $v$ with $\ell(u_i)<\ell(v)$ we let
  \[
  f(v)=(\ell(u_1),\dots,\ell(u_k),\ell(v),\dots,\ell(v)).
  \]
  Then it is easily checked that $f$ induces an isomorphism of $G$ and
  its image in~$\mathbf U$.
\end{proof}

\begin{lemma}
\label{lem:conj2}	
	Let $s,k$ be positive integers. 
	Let $E(G)=A\times B\setminus\bigcup_{i=1}^k E_i$, where each $E_i$ is a vertex-disjoint union of complete bipartite graphs. Then 
	there exists $D=D(s,k)$ such that either  $K_{s,s}$ is a subgraph of $G$ or $\delta(G)\leq D$.
\end{lemma}

\begin{proof}
We let $D=D(s,k)\coloneqq 2^{s2^{k}}$ and let
$G$ be as in the statement. By K\"onig's theorem, either $G$ 
contains a vertex cover $X$ with at most $D$ vertices or a matching~$M$ with at least $D$ edges. In the first case, $G$ contains a vertex
of degree at most~$D$ (consider a vertex not in $X$,
all its neighbors must be in $X$). 

In the second case, let $M=\bigl\{\{a_i,b_i\};\ i=1,\dots,D\bigr\}$.
Using Ramsey's theorem for pairs and $4$ colors we get that $M$ contains a matching $\bar M$, $|\bar M|\geq 2s$ such that for any two edges $\{a_i,b_i\}$ and $\{a_j,b_j\}$ in $\bar M$ the graph induced on the set $\{a_i,b_i,a_j,b_j\}$ is the same graph $H$. There are $4$ possibilities for $H$: either $|E(H)|=4$, or $|E(H)|=3$ (two possibilities for this case: either $\{a_i,b_j\}\in E(G)$ or $\{b_i,a_j\}\in E(G)$) or $|E(H)=2|$ (in which case $\bar M$ is an induced matching). In the first three cases, $K_{s,s}$ is a subgraph of $G$. The last case is impossible as by \cite{lovasz1980product} the complement of any graph containing an induced matching of size $D$ cannot covered by $\log D$ bipartite equivalences (i.e.\ disjoint unions of complete bipartite graphs). (In \cite[Proposition 5.3]{lovasz1980product} this is phrased in the language of the product dimension of the graph.)
	Hence it suffices to put $D=2^{s2^k}$ (we do not optimize here).
\end{proof}

\begin{theorem}
\label{thm:conj1}
	A class $\mathscr C$ is degenerate if and only if it is both weakly sparse and set-defined.
\end{theorem}
\begin{proof}
	If a class $\mathscr C$ is degenerate, then it is weakly sparse, and 
	it is set-defined by \Cref{lem:deg}.

	Conversely, let $\mathscr C$ be a weakly sparse set-defined class.
%  Assume that \Cref{conj:wssd} holds, and let $k,s$ be
%  integers. Consider the formula
%  \[
%  \phi(\bar x,\bar y):=\bigwedge_{i=1}^{k+1} \neg(x_i=y_i),
%  \]
%  where $|\bar x|=|\bar y|=k+1$, let $\mathbf U$ be the graph on
%  $\mathbb N^k$ and adjacency relation defined by $\phi$, and let
%  $\mathscr C=\{G\in{\rm Age}(\mathbf U)\mid K_{s,s}\not\subseteq
%  G\}$.
%  As the class $\mathscr C$ is weakly sparse and set-defined, it is
%  degenerate, hence there exists an integer $d$ such that every graph
%  in $\mathscr C$ has a vertex of degree at most $d$.  Now consider a
%  bipartite graph $G=(A,B,E)$ whose edge set $E$ can be written as
%  $E=(A\times B)\setminus(E_1\cup\dots\cup E_k)$ where, for each
%  $i=1,\dots,k$ the bipartite graph $G_i=(A,B,E_i)$ is a disjoint
%  union of bicliques, whose connected components are assumed to be
%  numbered.  For each vertex $v\in V(G)$ and integer
%  $i\in\{1,\dots,k\}$ let ${\rm cc}_i(v)$ be the number of the
%  connected component of $v$ in $G_i$, and let
%  $f:V(G)\rightarrow \mathbb N^{k+1}$ be defined by
%  $f(v)_i={\rm cc}_i(v)$ if $i\in\{1,\dots,k\}$ and $f(v)_{k+1}(v)$ is
%  $1$ if $v\in A$ and $2$ if $v\in B$. Then it is easily checked that
%  $f$ defines an isomorphism of $G$ and its image, which is an induced
%  subgraph of $\mathbf U$. It follows that either $K_{s,s}$ is a
%  subgraph of $G$ or $G\in\mathscr C$ thus $\delta(G)\leq d$. In
%  particular, \Cref{conj:2} holds.
%	
%	
%  Conversely, assume that \Cref{conj:wssd} does not hold. Then there
%  exists a non-degenerate weakly-sparse set-defined class
%  $\mathscr C$. 
	  %
  As $\mathscr C$ is weakly sparse, there exists an
  integer $s$ such that $K_{s,s}\not\subseteq G$ for all $G\in\Cc$.
  As $\mathscr C$ is set-defined, there is an integer
  $k$ and a formula $\phi(\bar x,\bar y)$ with $|\bar x|=|\bar y|=k$,
  such that every graph in $\mathscr C$ is an induced subgraph of the
  graph $\mathbf U$ with vertex set~$\mathbb N^k$ and adjacency
  relation defined by $\phi$.  We consider the formula
  $\phi'(\bar x',\bar y')$ with $|\bar x'|=|\bar y'|=k+1$ defined by
  \[
  \phi'(\bar x',\bar y'):=\phi(x_1,\dots,x_k,y_1,\dots,y_k)\wedge\neg
  (x_{k+1}=y_{k+1})
  \]	
  and we let $\mathbf U'$ be the corresponding graph definable on
  $\mathbb N^{k+1}$. 

\smallskip  
  	Assume towards a contradiction that $\mathscr C$ is not degenerate.
  For every graph $G\in\mathscr C$ and every
  bipartition $A,B$ of $V(G)$, the bipartite subgraph of $G$
  semi-induced by $A$ and $B$ belongs to the age of $\mathbf U'$.
  Moreover, if $G$ is not $d$-degenerate, then $G$ has such a
  semi-induced subgraph that is not $d/2$-degenerate. It follows that
  the class
  \[
  \mathscr C'=\{H\in{\rm Age}(\mathbf U')\mid K_{s,s}\not\subseteq
  H\text{ and $H$ bipartite}\}
  \]
  is also a counterexample in the sense that it is weakly sparse,
  set-defined and is not degenerate.  Considering a disjunctive normal
  form of $\phi'$, we get that there exists a family $\mathcal F$ of
  sets of pairs of integers in $\{1,\dots,k+1\}$ such that
  $\phi'(\bar x',\bar y')$ is logically equivalent to
  \[
  \bigvee_{P\in\mathcal F}\bigwedge_{(i,j)\in
    P}(x_i=y_j)\wedge\bigwedge_{(i,j)\notin P}\neg(x_i=y_j).
  \]
  
  This means that the edge set of the graphs in $\mathscr C'$ are the
  union of at most $(k+1)^2$ sets of edges, each being defined by a
  formula of the type
  \[
  \phi_{P}(\bar x',\bar y')=\bigwedge_{(i,j)\in
    P}(x_i=y_j)\wedge\bigwedge_{(i,j)\notin P}\neg(x_i=y_j).
  \]
  
  It follows that there exists $P\in\mathcal F$ such that the class
  $\mathscr C_P$ of the subgraphs of the graphs in $\mathscr C'$ with
  edge defined by $\phi_P$ is non-degenerate. It follows that we can
  assume without loss of generality that the formula $\phi'$ has the
  form
  $\bigwedge_{(i,j)\in P}(x_i=y_j)\wedge\bigwedge_{(i,j)\notin
    P}\neg(x_i=y_j)$
  for some set $P$ of pairs of integers in $\{1,\dots,k+1\}$. As the
  graphs we consider are bipartite we can consider an embedding in
  $\mathbb N^{(k+1)^2}$ instead of $\mathbb N^{k+1}$ by duplicating
  the $i$th coordinate of the vertices in the first part and the $j$th
  coordinate of the vertices in the second part to the coordinate
  $(i,j)$. This way we can assume that the formula has the form
  $\bigwedge_{i\in I}(x_i=y_i)\wedge\bigwedge_{i\notin
    I}\neg(x_i=y_i)$.
  As imposing $x_i=y_i$ allows only to create a disjoint union of
  induced subgraphs, it is useless in our setting. Thus we can assume
  that our graph is defined on some $\mathbb N^{k'}$ by the formula
  \[
  \psi(\bar x'',\bar y''):=\bigwedge_{i=1}^{k'} \neg(x_i=y_i).
  \]
  
  Thus our bipartite graphs contradict \Cref{lem:conj2}.
\end{proof}

\section{Order-defined classes}

%%%%%%%
We now consider a notion sandwiched between semi-algebraic and set-defined.
We call a class $\mathscr C$ {\em order-defined} if it is included in
the age of a graph definable in $(\mathbb Q,<)$, the countable dense
linear order without endpoints.  It is immediate that order-defined
classes are semi-algebraic. Both set-defined classes and order-defined classes are strongly related to the constructions introduced in \cite{trivial} for strongly polynomial sequences.

\begin{example}
  The class of circle graphs is order-defined. It follows that the
  class of cographs, and more generally the class of permutation
  graphs and the class of distance-hereditary graphs (which are both
  contained in the class of circle graphs) are order-defined.
\end{example}

However, we are not aware of any example of a semi-algebraic class that is not order-defined.

\smallskip
The notion of order-defined classes is very similar to the notion of Boolean dimension considered by Gambosi, Ne\v set\v ril, and Talamo \cite{GAMBOSI1990251} and by Ne\v set\v ril and Pudl\'ak \cite{Nesetril1989}. The {\em Boolean dimension} of a poset $P=(X,\leq)$ is the minimum number of linear orders on $X$ a Boolean combination of which gives $\leq$. Thus we have the following property:
\begin{observation}
	If a class $\mathcal P$ of posets has bounded Boolean dimension then the class of the comparability graphs of the posets in $\mathcal P$ is order-defined.
\end{observation}

The conjecture on boundedness of Ne\v set\v ril and Pudl\'ak \cite{Nesetril1989} on the boundedness of the Boolean dimension of planar posets can thus be weakened as follows.
\begin{conjecture}
	The class of comparability graphs of planar posets is order-defined.
\end{conjecture}

%%%%%%%

Every set-defined class has bounded order-dimension and is
order-defined. The converse might be true.

\begin{problem}
  Is every order-defined class with bounded order-dimension
  set-defined?
\end{problem}

%% file: covers.tex
%\section{About $2$-covers}
%\section{Transfer of some properties to $2$-covers}
\section{Inherited regularity of $2$-covered classes}
\label{sec:2cov}
In this section, we show that when a class $\mathscr C$ is $2$-covered by a class $\mathscr D$ we can deduce that $\mathscr C$ inherits many properties of the class $\mathscr D$, including some regularity properties.

\begin{theorem}
\label{thm:2cov}
	Assume $\mathscr D$ is a distal-defined (resp.\ semi-algebraic, order-defined, set-defined) class and that the class $\mathscr C$ is $2$-covered by $\mathscr D$. 
	Then the class $\mathscr C$ is also 
distal-defined (resp.\ semi-algebraic, order-defined, set-defined).
\end{theorem}
\begin{proof}
		Let $\mathbf U$ be a graph definable in a distal structure (resp.\ a semi-algebraic graph, a graph definable in $(\mathbb Q,<)$, a graph definable in $\mathbb N$) such that $\mathscr D\subseteq{\rm Age}(\mathbf U)$. 
		Assume that $\mathscr C$ is $2$-covered by $\mathscr D$ with magnitude $p$ and let $q\coloneqq \binom{[p]}{2}+1$. 
		We consider the graph $\mathbf H$ whose vertex set is the set of all tuples $\bar v\in V(\mathbf U)^q$ (we address the elements of the
tuples by indices $0$ and $\{i,j\}$ for $1\leq i<j\leq p$) 
and whose edges are defined by the formula %{\binom{[p]}{2}\cup\{0\}}
\[
\phi(\bar x,\bar y):=
		\bigwedge_{1\leq i<j\leq p} (x_{\{i,j\}}=x_0)\vee (y_{\{i,j\}}=y_0)\vee E(x_{\{i,j\}},y_{\{i,j\}}).
\]
		Let $G\in\mathscr C$. By assumption there exists a partition $V_1,\dots,V_p$ of $V(G)$ such that for every $1\leq i<j\leq p$ we have $G[V_i\cup V_j]\in\mathscr D$.
		We denote by $f_{i,j}$ the embedding of $G[V_i\cup V_j]$ in $\mathbf U$.
Let $g\colon V(G)\rightarrow V(\mathbf U)^q$ be defined as follows: 
let $u\in V_i$ and let $z(u)$ be an arbitrary vertex of $V(\mathbf U)\setminus \{f_{i,j}(u): j\in [p]\setminus\{i\}\}$.
We let $g(v)=\bar x$, where $x_0=z(u)$ and 
\[
	x_{\{k,\ell\}}=\begin{cases}
		z(u)&\text{if }i\notin\{k,\ell\}\\
		f_{\{k,\ell\}}(u)&\text{otherwise}
	\end{cases}
\]

Then it is easily checked that $g$ induces an isomorphism of $G$ and its image in $\mathbf H$.	 As $\mathbf H$ is definable in $\mathbf U$ we infer that the class $\mathscr C$ is  
distal-defined (resp. semi-algebraic, order-defined, set-defined).
\end{proof}

\smallskip
\begin{corollary}
\label{crl:rwcovers-setdefined}
  A class has structurally bounded expansion if and only if it has low
  linear rankwidth covers and is set-defined. 
\end{corollary}
\begin{proof}
  Every class $\Cc$ with structurally bounded expansion has low shrubdepth
  covers (by \Cref{thm:sbe}), thus, has low linear rankwidth covers. 
  As classes with bounded shrubdepth are set-defined 
  by~\Cref{lem:sd-setdef}, it follows from \Cref{thm:2cov} that
  $\Cc$ is set-defined. 
  
  \smallskip
  Conversely, assume $\mathscr C$ has low linear rankwidth covers and
  is set-defined. As~$\mathscr C$ is set-defined, it has bounded
  order-dimension. According to \cite{msrw}, a graph with bounded
  linear rankwidth covers and bounded order-dimension is a transduction of a
  graph with bounded pathwidth covers. It follows that $\mathscr C$ has
  structurally bounded expansion.
\end{proof}

The following theorem is possibly the first purely model theoretical characterization of bounded expansion classes.
\begin{theorem}
	A hereditary class $\mathscr C$ has bounded expansion if and only if $\Cc$ is weakly sparse 
	%neither contains all cliques nor all bicliques, 
	and if all transductions of $\mathscr C$ are set-defined.
\end{theorem}
\begin{proof}
	If $\mathscr C$ has bounded expansion, then it is weakly sparse and 
	all transductions of $\mathscr C$  are set-defined by \Cref{crl:rwcovers-setdefined}.

\smallskip	
Conversely, assume towards a contradiction that $\Cc$ is weakly
sparse and all transductions of $\mathscr C$ are set-defined and that 
$\Cc$ fails to have bounded expansion. By \cite{dvovrak2018induced} there exists an integer $p$ such that $\mathscr C$ includes the $p$-subdivision of graphs with arbitrarily large average degree. According to \cite{Kuhn2004}, 
$\mathscr C$ also includes the $p$-subdivision of $C_4$-free graphs with arbitrarily large average degree. By an easy transduction we obtain from $\mathscr C$ a class $\mathscr D$ of $C_4$-free graphs with unbounded average degree. This class $\mathscr D$ being weakly sparse and (by assumption) set-defined, contradicting \Cref{thm:conj1}.
\end{proof}
%
%\textcolor{blue}{
%\Cref{conj:wssd} would imply that set-defined nowhere dense classes
%are exactly classes with bounded expansion. 
%One may ask whether this
%could be extended from nowhere dense and bounded expansion to
%structurally nowhere dense and structurally bounded expansion.
%}
%\begin{problem}
%  Is it true that a structurally nowhere dense class is set-defined if
%  and only if it has structurally bounded expansion?
%\end{problem}

\smallskip
\begin{theorem}
 	Assume $\mathscr D$ is a class with bounded VC-dimension (resp.\ bounded order-dimension) and that the class $\mathscr C$ is $2$-covered by $\mathscr D$. 
	Then the class $\mathscr C$ also has bounded VC-dimension (resp.\ bounded order-dimension).
\end{theorem}
\begin{proof}
	Let ${\rm ES}_n$ be the bipartite graph with vertex set $[n]\cup 2^{[n]}$ and adjacency given by $\{i,I\}\in E({\rm ES}_n)$ if $i\in I$.
	Consider any $p$-coloring of the vertices of ${\rm ES}_n$. Let $A_1,\dots,A_p$ be the color classes of $[n]$ and let $B_1,\dots,B_p$ be the color classes of $2^{[n]}$. We have
\[
	2^n=|\{N(b)\mid b\in 2^{[n]}\}|=\sum_{i=1}^p |\{N(b)\mid b\in B_i\}|\leq \sum_{i=1}^p\sum_{j=1}^p |\{N(b)\cap A_j\mid b\in B_i\}|
\]
	Hence there exists $i,j\in [p]$ such that 
	$|\{N(b)\cap A_j\mid b\in B_i\}|\geq 2^n/p^2$. It follows from the Sauer-Shelah lemma that ${\rm ES}_n$ contains a bichromatic induced subgraph with VC-dimension $\Omega(\frac{1}{\log p}\cdot\frac{n}{\log n})$. In particular, if a class $\mathscr C$ has unbounded VC-dimension and is $2$-covered by a class $\mathscr D$, then $\mathscr D$ has unbounded VC-dimension.

\smallskip	
	Now consider any $p$-coloring of the half-graph $H_n$ with vertices $a_1,\dots,a_n$ and $b_1,\dots,b_n$ where $a_i$ is adjacent to $b_j$ if $i\leq j$.
	By the pigeon-hole principle, there exists a pair of colors $(c_1,c_2)$ and a subset $I\subseteq [n]$ of size at least $n/p^2$ such that for every $i\in I$ the vertex $a_i$ has color $c_1$ and the vertex $b_i$ has color $c_2$. It follows that $H_n$ contains a bi-chromatic induced $H_{n/p^2}$. Consequently, if a class $\mathscr C$ has unbounded order-dimension and is $2$-covered by a class $\mathscr D$, then $\mathscr D$ has unbounded order-dimension.
\end{proof}

%%%%%%%%%%%%%%%%%
\smallskip
Recall the definition of $\epsilon$-nice partitions from \Cref{def:nice-decomposition}: A partition $V_1,\ldots,V_k$ of the vertex set of a graph $G$ is $\epsilon$-nice if 
	 \[
    \sum_{\text{non-homogenous }(V_i,V_j)}\frac{|V_i|\,|V_j|}{n^2}<\epsilon.
    \]

\smallskip
\begin{theorem}[Regularity preservation for $2$-covered classes]
\label{thm:reg2cov}
Let $\mathscr D$ be a class such that for every $\epsilon>0$ every 
graph in $\mathscr D$ has an $\epsilon$-nice partition with $f(\epsilon)$ parts (for some function $f$).
	Let $\mathscr C$ be a class $2$-covered by $\mathscr D$ with magnitude~$p$.
	Then, for every $\epsilon>0$ every graph $G\in\mathscr C$ has an $\epsilon$-nice partition with $K\leq pf(\frac{\epsilon}{p-1})^{p-1}$ parts.
\end{theorem}

\begin{proof}
	Assume $V(G)=\bigcup_{i=1}^p V_i$ and $G[V_i\cup V_j]\in\mathscr D$ for all $1\leq i<j\leq p$.
	Let $G_{i,j}=G[V_i\cup V_j]$. According to the assumptions, there is an $\frac{\epsilon}{p-1}$-nice partition $\mathcal{P}_{i,j}=(W_1^{i,j},W_2^{i,j},\cdots,W_{K_{i,j}}^{i,j})$, $K_{i,j}\leq f(\frac{\epsilon}{p-1})$, of $V(G_{i,j})$ satisfying
	\[
    \sum_{\text{non-homogeneous }(W_s^{i,j},W_t^{i,j})
    }\frac{|W_s^{i,j}|\,|W_t^{i,j}|}{|V(G_{i,j})|^2}<\frac{\epsilon}{p-1}.
	\]

	For convenience we define $W_{j,i}=W_{i,j}$.
	For $i\in [p]$ let 
\[
	\mathcal I_i=[K_{i,1}]\times\dots [K_{i,i-1}]\times\{0\}\times [K_{i,i+1}]\times [K_{i,p}]
\]
	
	Define a partition $\mathcal P=(P^i_{\alpha})_{1\leq i\leq p,  \alpha\in \mathcal I_i}$ by letting
\[
	P^i_{\alpha}=\bigcap_{j\neq i}W^{i,j}_{\alpha_j}.
\]
%	
%	Define a new partition $W_1,W_2,\cdots,W_K$ of $G$ such that 
%	for every $t\in [K]$, $W_t=\cap_{1\le i,j\le p} P^{i,j} $, where $P^{i,j}$ is a part of $\mathcal{P}^{i,j}$.

	Then 
	\[
	|\mathcal P|\le p\,f\Bigl(\frac{\epsilon}{p-1}\Bigr)^{p-1}.
	\]
	
	As cutting homogeneous pairs gives only homogeneous pairs, the only non homogeneous pairs are pairs 
	$(P^i_{\alpha},P^j_{\beta})$, where $\alpha\in\mathcal I_i$, $\beta\in\mathcal I_j$ and 
	$(W^{i,j}_{\alpha_j},W^{i,j}_{\beta_i})$ is not homogeneous.
	Thus we have
\begin{align*}
	&\sum_{\text{non-homogeneous }(P^i_\alpha,P^j_\beta)}\frac{|P^i_\alpha|\,|P^j_\beta|}{n^2}\\
	&\qquad\leq\sum_{\text{non-homogeneous }(W^{i,j}_s,W^{i,j}_t)}
	\sum_{\alpha\in I_i: \alpha_j=s}\sum_{\beta\in I_j: \beta_i=t}
	\frac{|P^i_\alpha|\,|P^j_\beta|}{n^2}\\
	&\qquad\leq \sum_{\text{non-homogeneous }(W^{i,j}_s,W^{i,j}_t)}\frac{|W^{i,j}_s|\,|W^{i,j}_t|}{n^2}\\
	&\qquad\leq \sum_{\text{non-homogeneous }(W^{i,j}_s,W^{i,j}_t)}\frac{|G_{i,j}|^2}{n^2}\frac{|W^{i,j}_s|\,|W^{i,j}_t|}{|G_{i,j}|^2}\\
	&\qquad \leq \biggl(\sum \frac{|G_{i,j}|^2}{n^2}\biggr)\frac{\epsilon}{p-1}<\epsilon.
\end{align*}
%\begin{align*}
%	&\sum_{\text{non homogeneous }(P^i_{k_1,\dots,k_{i-1},k_{i+1},\dots,k_p},P^j_{\ell_1,\dots,\ell_{j-1},\ell_{j+1},\dots,\ell_p})}\frac{|P^i_{k_1,\dots,k_{i-1},k_{i+1},\dots,k_p}|\, |P^j_{\ell_1,\dots,\ell_{j-1},\ell_{j+1},\dots,\ell_p}|}{n^2}\\
%	&\qquad\leq \sum_{\text{non homogeneous }(W^{i,j}_k,W^{i,j}_\ell)}
%	\sum_{\mathbf k\text{ with }k_j=k}}\sum_{\mathbf\ell
%	\text{ with }\ell_i=\ell}\frac{|P^i_{\mathbf k}|\, |P^j_{\mathbf \ell}|}{n^2}
%\end{align*}
%	According to the assumptions, $W_s$ and $W_t$ ($s,t\in [K]$, $s\ne t$) are not homogeneous only if $W_s\subseteq V_i$, $W_t\subseteq V_j$, $1\le i<j\le p$. It implies that
%	\[
%	\begin{split}
%	&\sum_{\text{non homogeneous }(W_s,W_t)}\frac{|W_s|\,|W_t|}{n^2}\\
%	<&\frac{1}{p-1}\sum_{1\le i<j\le p}\sum_{\text{non homogeneous }(W_s^{i,j},W_t^{i,j}) }\frac{|W_s^{i,j}|\,|W_t^{i,j}|}{n^2}\\
%	<&\frac{1}{p-1}\sum_{1\le i<j\le p}\frac{|V(G_{i,j})|^2}{n^2}\sum_{\text{non homogeneous }(W_s^{i,j},W_t^{i,j}) }\frac{|W_s^{i,j}|\,|W_t^{i,j}|}{|V(G_{i,j})|^2}\\
%	<&\frac{1}{p-1}\sum_{1\le i<j\le p}\frac{|V(G_{i,j})|^2}{n^2}\epsilon\\
%	%<&\frac{1}{p-1}\frac{(p-1)n^2}{n^2}\epsilon\\
%	<&\epsilon.
%	\end{split}
%	\]   
\end{proof}

The following corollary is then a direct consequence of \Cref{thm:2cov} and \Cref{lem:equipartition}.

\begin{corollary}
\label{cor:reg2cov_equi}
Let $\mathscr D$ be a class such that for every $\epsilon>0$ every 
graph in $\mathscr D$ has an $\epsilon$-nice partition with $f(\epsilon)$ parts (for some function $f$).
	Let $\mathscr C$ be a class $2$-covered by $\mathscr D$ with magnitude~$p$.
	Then, for every $\epsilon>0$ every graph $G\in\mathscr C$ has an equipartition with $K\leq \frac{2p}{\epsilon}f(\frac{\epsilon}{2(p-1)})^{p-1}$ parts such that all pairs but an $\epsilon$-fraction are homogenous
\end{corollary}

%% file: cographs.tex
\section{Regularity for classes $2$-covered by embedded $m$-partite cographs}
\label{sec:cog}
In this section, we show that 
 classes $2$-covered by a class of  embedded $m$-partite cographs 
 are order-defined (\Cref{cor:emb-2cov-sa}), hence, satisfy the semi-algebraic regularity lemma (\Cref{thm:sa}).
 Moreover, we give an explicit construction  of an $\epsilon$-nice partition with explicit bound for the number of parts (\Cref{cor:reg_emb_2cov})  in the style of the regularity lemma for distal-defined classes (\Cref{thm:distal_reg}).

\smallskip
A {\em cograph}, or complement-reducible graph, is a graph that can be generated from $K_1$ by complementations and disjoint unions. The tree representation of a cograph $G$ is a rooted tree $T$ (called {\em cotree}), whose leaves are the vertices of $G$ and whose internal nodes represent either disjoint unions or complete joins operations. Some generalizations of cographs have been proposed; e.g.\ bi-cographs~\cite{Giakoumakis1997}, $k$-cographs~\cite{Hung2011}, or $m$-partite cographs~\cite{Ganian2012}. The following extension we present here is very natural:

\begin{definition}
  An {\em embedded $m$-partite cograph} is a graph that can be
  obtained from a plane tree $T$ with a coloring
  $\gamma_L\colon L(T)\rightarrow \{1,\dots,m\}$ and an assignment
  $v\in I(T) \mapsto f_v$, with
  $f_v\colon [m]\times [m]\rightarrow \{0,1\}$, as follows: the vertex
  set of $G$ is~$L(T)$ and two vertices $u$ and $v$ are adjacent if
  (assuming that the branch from~$u\wedge v$ to~$u$ is to the left of
  the branch from $u\wedge v$ to $v$) we have
  $f_{u\wedge v}(\gamma_L(u),\gamma_L(v))=1$.
\end{definition}
By similarity with the case of cographs, the colored plane tree $T$ is
called an {\em embedded cotree} of $T$.

\smallskip
Embedded $m$-partite cographs share some nice properties with $m$-partite cographs: for instance, they are well quasi-ordered for induced subgraph inclusion, and they have bounded rankwidth, which makes their recognition fixed-parameter tractable. Moreover, they are linearly $\chi$-bounded, as for every embedded $m$-partite cograph $G$ the $m$ color classes of $G$ induce cographs thus $\chi(G)\leq m \omega(G)$.

\smallskip
Embedded $m$-partite cographs also generalize 
graphs with bounded shrubdepth~\cite{Ganian2012} and  graphs with
bounded embedded shrubdepth~\cite{SODA_msrw}. It follows that classes $2$-covered by a class of embedded $m$-partite cographs include structurally bounded expansion classes \cite{SBE_drops} and, more generally, class $2$-covered by a class with bounded linear rankwidth \cite{SODA_msrw}. As an example, this includes the class of unit-interval graphs \cite{SODA_msrw,msrw}.

\smallskip
Although embedded $m$-partite cographs fail to be set-defined in general (as witnessed by cographs), we have the following.

\begin{lemma}
\label{lem:empc_order}
  The class of embedded $m$-partite cographs is order-defined.
\end{lemma}
\begin{proof}
  Consider an embedded cotree $T$ of an embedded $m$-partite cograph
  $G$. Then the subgraph of $G$ induced by any two colors is obviously
  an embedded $2$-partite cograph. Hence, the class of embedded
  $m$-partite cograph is $2$-covered by the class of $2$-partite
  cographs. To each embedded cotree $T$ of an embedded $2$-partite
  cograph we associate the linear order $<_T$ on the leaves of $T$
  defined by a left-to-right traversal of $T$ and two cotrees $T_1$
  and $T_2$ (on the same rooted tree as $T$), where an internal node
  $x$ defines a complete join in $T_1$ (resp. $T_2$) if $f_x(1,2)=1$
  (resp. $f_x(2,1)=1$) and a disjoint union, otherwise.  These cotrees
  define two graphs $G_1$ and $G_2$ on the vertex set of $G$. Let
  $M_1$ and $M_2$ mark the vertices with color $1$ and color $2$,
  respectively.  Then if $u$ has color $1$ and $v$ has color $2$, the
  vertices $u$ and $v$ are adjacent in $G$ if and only if $u<_T v$ and
  $u$ and $v$ are adjacent in $G_1$, or $u>_T v$ and $u$ and $v$ are
  adjacent in $G_2$.  Let $g_1$ (resp.\ $g_2$) be an embedding of
  $G_1$ (resp.\ $G_2$) in a graph definable in $(\Q,<)$, and let
  ${\rm index}$ be the index of vertices in the linear order
  $<_T$. Then to each vertex $v$ we associate the vector
  $(1,{\rm color}(v), {\rm index}(v), g_1(v),g_2(v))$. It is easily
  checked that this allows to define embedded $2$-partite cographs as
  induced subgraph of a graph definable in $(\Q,<)$.
\end{proof}

\smallskip
Using \Cref{thm:reg2cov} we deduce:
\begin{corollary}
  Every class $\mathscr C$ that is $2$-covered by a class of embedded
  $m$-partite cographs is order-defined.
\end{corollary}

As each class with bounded linear rankwidth is $2$-covered by a class
of embedded $m$-partite cographs (and even by a class of bounded
embedded shrubdepth \cite{SODA_msrw}) we deduce

\smallskip
\begin{corollary}
\label{cor:emb-2cov-sa}
  Every class $\mathscr C$ that is $2$-covered by a class with bounded
  linear rankwidth is also $2$-covered by a class of embedded
  $m$-partite cographs, thus order-defined.
\end{corollary}

\smallskip
Note that classes $2$-covered by a class with bounded linear rankwidth
include in particular structurally bounded expansion classes, as these
classes admit low shrubdepth covers \cite{SBE_drops}.

\medskip We now give an explicit construction and proof for a
weakened statement of the distal regularity lemma in the special case
of embedded $m$-partite cographs.

\smallskip
\begin{definition}
  A {\em plane tree} is a rooted tree in which the children of each
  vertex are ordered from left to right.
\end{definition}

\smallskip
Each plane tree $T$ with root $r$ defines a partial order $\leq$ on
its vertex set by $u\leq v$ if the path linking $r$ to $v$ in $T$ goes
through $u$.  For vertices $u,v$ in $V(T)$ we denote by $u\wedge v$
the least common ancestor of $u$ and $v$ in $T$, that is the maximum
$z\in V(T)$ with $z\leq u$ and $z\leq v$.  We denote by $L(T)$ the set
of all leaves of $T$ and by $I(T)$ the set $V(T)\setminus L(T)$ of all
internal nodes of $T$.  For a vertex $v\in V(T)$ we denote by $T_v$
the plane subtree of $T$ rooted at $v$. We also denote by $F_v$ the
ordered set of the children of $v$ in $T$.  We shall consider two
coloring functions on~$V(T)$,
$\gamma_L\colon L(T)\rightarrow \{1,\dots,c_L\}$ and
$\gamma_I\colon I(T)\rightarrow \{1,\dots,c_I\}$.

\smallskip
\begin{definition}
  Let $T$ be a plane tree, let $\mu$ be a probability measure on
  $V(T)$, and let $\epsilon\geq \max_{v\in V(T)}\mu(v)$.  Then
  $v\in T$ is called
  \begin{itemize}
  \item {\em $\epsilon$-light} if $\mu(T_v) \leq \epsilon$;
  \item {\em $\epsilon$-terminal} if $v$ is not $\epsilon$-light but
    all the children of $v$ are $\epsilon$-light;
  \item {\em $\epsilon$-singular} if $v$ has exactly $1$ child that is
    not $\epsilon$-light and the sum of the $\mu$-measures of the
    $T_u$ for $\epsilon$-light children $u$ of $v$ is strictly greater
    than $\epsilon$;
  \item {\em $\epsilon$-chaining} if $v$ has exactly $1$ child that is
    not $\epsilon$-light and $v$ is not $\epsilon$-singular;
  \item {\em $\epsilon$-branching} if $v$ has at least $2$ children
    that are not $\epsilon$-light.
  \end{itemize}
\end{definition}

\smallskip
\begin{definition}
\label{def:partition}
Let $T$ be a plane tree, let $\mu$ be a probability measure on $V(T)$,
and let $\epsilon\geq \max_{v\in V(T)}\mu(v)$.  A partition
$\mathcal{P}$ of $V(T)$ is an {\em $\epsilon$-partition} of $T$ if
\begin{itemize}
\item each part is of one of the following types:
  \begin{enumerate}
  \item\label{typ1} $P=\{v\} \cup \bigcup_{x \in F} T_x$ for some
    non-empty interval $F\subseteq F_v$,
  \item\label{typ2} $P=\bigcup_{x \in F} T_x$ for some interval
    $F \subseteq F_v$,
  \item\label{typ3} $P=T_v \setminus T_w$ for some $w\in T_v$ distinct
    from $v$ ($w$ is called the {\em cut vertex} of $P$, and the path
    from $v$ to the father of $w$ is called the {\em spine} of $P$),
  \end{enumerate}
  where $v$ (which is easily checked to be the infimum of $P$) is
  called the {\em attachement vertex} of $P$ and is denoted by
  $A_{\mathcal P}(P)$;
\item each attachment vertex of a part of type~\ref{typ2} is also the
  attachment vertex of some part of type~\ref{typ1};
\item every part has $\mu$-measure at most $\epsilon$.
\end{itemize}
\end{definition}

We now prove that every plane tree has a small $\epsilon$-partition.

\begin{lemma}
  \label{lem}
  Let $T$ be a plane tree, let $\mu$ be a probability measure on
  $V(T)$, and let $\epsilon > 0$. We assume that no vertex has measure
  more than $\epsilon$.  Then there exists an $\epsilon$-partition
  $\mathcal{P}$ of $T$ with $8/\epsilon$ vertices.
\end{lemma}
\begin{proof}
  When considering a partition of $V(T)$, a part $X$ is {\em thin} if
  $\mu(X)\leq\epsilon$ and {\em thick} if $\mu(X)>\epsilon$. Note that
  the number of thick parts in a partition is at most~$1/\epsilon$ (as
  they are disjoint and have global measure at most $1$).

  We define a first partition $\mathcal P_0$ of $V(T)$ into {\em
    atoms}, where an atom is
  \begin{itemize}
  \item $T_v\setminus T_w$ if $v$ is $\epsilon$-chaining, the atom
    corresponding to $v$ has measure at most $\epsilon$, and $w$ is
    the unique child of $v$ that is not $\epsilon$-light,
  \item $T_v$ if $v$ is $\epsilon$-light and the parent of $v$ is not
    $\epsilon$-light,
  \item or $\{v\}$ if $v$ is neither $\epsilon$-branching nor
    $\epsilon$-light.
  \end{itemize}

  We then define a coarser partition $\mathcal P_1$, in which atoms
  are gathered into groups, where a group is of the following types:
  \begin{enumerate}
  \item $T_v\setminus T_w$ where $v$ is $\epsilon$-chaining and $w$ is
    the (unique) maximum descendant of $v$ that is neither
    $\epsilon$-light nor $\epsilon$-chaining.
  \item the union of $\{u\}$ (for non $\epsilon$-light and non
    $\epsilon$-chaining $u$) and the maximal (possibly empty) union of
    $T_{v_i}$, where the $v_i$'s are consecutive $\epsilon$-light
    children of $u$ starting from the leftmost one;
  \item a maximal union of $T_{v_i}$, where the $v_i$'s are
    consecutive $\epsilon$-light children of a non $\epsilon$-light
    and non $\epsilon$-chaining vertex $u$ not including the leftmost
    $\epsilon$-light child of $u$.
  \end{enumerate}

  Denote by $n_s$ the number of singular vertices of $T$, by $n_t$ the
  number of terminal vertices of $T$, and by $n_b$ the number of
  branching vertices $T$.  All the atoms corresponding to
  $\epsilon$-terminal or $\epsilon$-singular vertices are disjoint and
  thick thus $n_t+n_s\leq 1/\epsilon$.

  We consider a reduced plane tree $T_1$ obtained from $T$ by removing
  all $\epsilon$-light vertices and contacting all $\epsilon$-chaining
  groups. Note that the leaves of $T_1$ are exactly the
  $\epsilon$-terminal vertices of $T$ hence $T_1$ has $n_t$ leaves. As
  every $\epsilon$-branching vertex of $T$ corresponds to a vertex of
  $T_1$ with at least two children, we have~$n_b<n_t$. The number
  $n_1$ of groups of type (1) is less than the number of edges of
  $T_1$ hence less than $n_t+n_s+n_b$. The number $n_2$ of groups of
  type (2) is at most $n_t+n_s+n_b$ (direct from the definition). The
  number $n_3$ of groups of type (3) is at most the number of edges of
  $T_1$ hence less than $n_t+n_s+n_b$. Altogether, we have
  $|\mathcal P_1|\leq 3(n_t+n_s+n_b)<6n_t+3n_s<6/\epsilon$.

  The partition $\mathcal P$ is sandwiched between $\mathcal P_1$ and
  $\mathcal P_0$. It is obtained by splitting thick groups into
  maximal parts of consecutive atoms with global measure at most
  $\epsilon$. If a group $X$ is split into $P_1,\dots,P_k$ then, by
  maximality we have $\mu(P_i)+\mu(P_{i+1})>\epsilon$ for every $i$ in
  $1,\dots,k-1$. Summing up we get
  $\mu(P_1)+2\mu(P_2)+\dots+2\mu(P_{k-1})+\mu(P_k)>(k-1)\epsilon$ thus
  $k-1<2\mu(X)/\epsilon$.  Summing over all thick parts, we get
  $|\mathcal P|-|\mathcal P_1|<2/\epsilon$.  Thus
  $|\mathcal P|<8/\epsilon$.
\end{proof}

\begin{theorem}[Regularity lemma for embedded $m$-partite cographs]
  \label{thmp}
  For every $\epsilon>0$, every (sufficiently large) embedded
  $m$-partite cograph $G$ of order $n$ has a vertex partition into
  $V_1,\dots,V_\ell$ with
  \mbox{$\ell\leq\frac{128}{\epsilon}m2^{m^2}=O\bigl(\frac{1}{\epsilon}\bigr)$},
  such that
  \[
  \sum_{\text{non-homogeneous
    }(V_i,V_j)}\frac{|V_i|\,|V_j|}{n^2}<\epsilon.
  \]
\end{theorem}

\begin{proof}
  We consider the embedded cotree $T$ of $G$ with two coloring
  functions on~$V(T)$,
  $\gamma_L\colon L(T)\rightarrow\{1,2,\cdots,c_L\}$ and
  $\gamma_I\colon I(T)\rightarrow \{1,2,\cdots,c_I\}$, where $c_I=m$
  and $c_L\le 2^{m^2}$. Let $\mu$ be a probability measure on $V(T)$
  such that for every vertex $v\in T$,
  \[
  \mu (v)=\begin{cases}
    \frac{1}{|V(G)|}, \text{if} \ v\in L(T);\\
    \ 0, \ \text{otherwise}.
	\end{cases}
  \]
	
  We say a partition $\mathcal{P}^*$ of $G$ is a {\em refinement} of a
  partition $P$ of $T$, if
  \begin{itemize}
  \item every part is consecutive.
  \item for every part $P^*$ of $\mathcal{P}^*$, there exists a part
    $P$ of $\mathcal{P}$ such that $P^*\subseteq P$.
  \item for each pair of parts $P^*_i\subseteq P$ and
    $P^*_j\subseteq P'$ ($P\ne P'$),
    \[|\{\gamma_I(u\wedge v)\mid (u,v)\in P^*_i\times P^*_j\}|=1.\]
  \item each part $P^*$ has $|\gamma_L(P^*)|=1$.
  \end{itemize}

  Now we consider a refinement $\mathcal{P}^*=(V_1,V_2,\cdots,V_l)$ of
  an $\frac{\epsilon}{8}$-partition $\mathcal P$.  By \Cref{lem},
  $|\mathcal{P}|<\frac{64}{\epsilon}$. Then
  $l\le 2|\mathcal{P}|c_Ic_L<\frac{128}{\epsilon}m2^{m^2}$.
   
  Observe that $V_i$ and $V_j$ are homogeneous for any
  $V_i\subseteq P$, $V_j\subseteq P'$ ($P\ne P'$) by definition of
  $\mathcal{P}^*$. Then the only possible case that $V_i$ and $V_j$
  are not homogeneous if $V_i,V_j\subseteq P$, which implies that
  \[
    \sum_{\text{non-homogeneous
      }(V_i,V_j)}\mu(V_i)\mu(V_j)=\sum_{\text{non-homogeneous
      }(V_i,V_j)}\frac{|V_i|\,|V_j|}{n^2}\le
    \Bigl({\frac{\epsilon}{8}}\Bigr)^2\,|\mathcal{P}|<\epsilon.
  \]
\end{proof}

The next corollary then follows from the application of \Cref{thm:2cov}.
\begin{corollary}[Regularity lemma for classes $2$-covered by embedded $m$-partite cographs]
\label{cor:reg_emb_2cov}
Let $\mathscr D$ be a class of embedded $m$-partite cographs and
	let $\mathscr C$ be a class $2$-covered by $\mathscr D$ with magnitude~$p\geq 2$.
	Then, for every $\epsilon>0$ every graph $G\in\mathscr C$ has an $\epsilon$-nice partition with at most $p\Bigl(\frac{128m2^{m^2}(p-1)}{\epsilon}\Bigr)^{p-1}$ parts.
\end{corollary}

Applying \Cref{lem:equipartition} we also get

\begin{corollary}
\label{cor:reg_emb_2cov_equi}
Let $\mathscr D$ be a class of embedded $m$-partite cographs and
	let $\mathscr C$ be a class $2$-covered by $\mathscr D$ with magnitude~$p\geq 2$.
	Then, for every $\epsilon>0$ every graph $G\in\mathscr C$ has an equipartition into at most $2p\bigl({m2^{m^2+8}(p-1)}\bigr)^{p-1}{\epsilon}^{-p}$ parts, such that all pairs but an $\epsilon$-fraction are homogenous.
\end{corollary}

Note that classes $2$-covered by a class of embedded $m$-partite cographs include classes $2$-covered by a class with bounded linear rankwidth \cite{SODA_msrw} and structurally bounded expansion classes (as follows from \cite{SBE_drops}).

A natural question is whether the distal regularity lemma applies for classes for bounded rankwidth (and thus to classes $2$-covered by a class with bounded rankwidth). The following problem would be a way to get a positive answer.

\begin{problem}
\label{pb:rw}
	Is every class with bounded rankwidth distal-defined?
\end{problem}

%% file: nd.tex
%\section{Nowhere dense classes}
\section{Regularity and non-regularity for nowhere dense classes}
\label{sec:nd}
In this section, we study regularity properties of nowhere dense classes, especially through the regularity properties of the $d$-powers of the graphs in the class (\Cref{thm:ndreg}). On a negative side, we prove that there exists a nowhere dense class that not only fails to be distal-defined, but also does not allow any $\epsilon$-nice partition (\Cref{cor:nd-noreg}).

\begin{proposition}
  For a weakly sparse hereditary class $\mathscr C$ the following are
  equivalent:
  \begin{enumerate}[{\rm (i)}]
  \item\label{enum:nd1} $\mathscr C$ is nowhere dense,
  \item\label{enum:nd2} for every integer $d\geq 1$ there exists $n$
    such that $\mathscr C$ excludes the $d$-subdivision of $K_n$ as an
    induced subgraph,
  \item\label{enum:nd3} for every integer $d$ the hereditary closure
    of the class $\mathscr C^d=\{G^d\mid G\in\mathscr C\}$ does not
    contain all graphs,
  \item\label{enum:nd4} for every integer $d$ the class
    $\mathscr C^d=\{G^d\mid G\in\mathscr C\}$ has bounded order
    dimension.
  \end{enumerate}
\end{proposition}
\begin{proof}
  \eqref{enum:nd1}$\iff$\eqref{enum:nd2} was proved by Dvo\v r\'ak
  \cite{dvovrak2018induced}.
  \eqref{enum:nd3}$\Rightarrow$\eqref{enum:nd2} (by contrapositive):
  assume that for some integer $d\geq 1$ the class $\mathscr C$
  contains the $d$-subdivision of all complete graphs, then it also
  contains the $d$-subdivision of all graphs, hence the hereditary
  closure of $\mathscr C^{d+1}$ contains all
  graphs. \eqref{enum:nd4}$\Rightarrow$\eqref{enum:nd3} as for each
  integer $d$ some half-graph is not in the hereditary closure of
  $\mathscr C^d$.  \eqref{enum:nd1}$\Rightarrow$\eqref{enum:nd4} by
  \Cref{crl:nd-stable}.
\end{proof}

The following theorem gives yet another characterization of nowhere dense classes.

\begin{theorem}[Nowhere dense by regularity lemma]
\label{thm:ndreg}
  Let $\mathscr C$ be a hereditary class of graphs.  Then the
  following are equivalent:
  \begin{enumerate}[{\rm (i)}]
    \item \label{enum:sp3} the class $\mathscr C$ is nowhere dense;
  \item \label{enum:sp2} For every $d,K\in\mathbb N$ and $\epsilon>0$
    there exists an integer $s$ such that for every graph
    $G\in\mathscr C$ there exists a subsets $S\subseteq V(G)$ of size
    at most $s$ and a equipartition $V_1,\dots,V_K$ of
    $V(G)\setminus S$ with the property that for every
    $1\leq i\leq K$, each vertex $u\notin S$ is at distance more than
    $d$ than at least $(1-\epsilon)$ proportion of $V_i$ in $G-S$;
      \item \label{enum:sp1} For every $d\in\mathbb N$ and $\epsilon>0$
    there exist integers $K$ and $s$ such that for every graph
    $G\in\mathscr C$ there exists a subsets $S\subseteq V(G)$ of size
    at most $s$ and an equipartition $V_1,\dots,V_K$ of
    $V(G)\setminus S$ with the property that for every
    $1\leq i,j\leq K$, each vertex $u$ of a subset of $V_i$ of size at
    least $(1-\epsilon)|V_i|$ is at distance more than $d$ than at
    least $(1-\epsilon)$ proportion of $V_j$ (which may depend on $u$)
    in $G-S$.
  \end{enumerate}
\end{theorem}

\begin{proof}
  \eqref{enum:sp3}$\Rightarrow$\eqref{enum:sp2}: By
  \cite[Theorem 42]{SurveyND} there exists an
  integer $s$ such that for every graph $G\in\mathscr C$ there exists
  a subset $S$ of at most $s$ vertices such that no ball of radius $d$
  in $G-S$ contains more than an $\epsilon/K$ proportion of the
  vertices.  This means that the maximum degree of $(G-S)^d$ is at
  most $(\epsilon/K)|G|$. Consider any equipartition of $V(G)-S$ into
  $K$ classes $V_1,\dots,V_K$. Then the degree of every vertex of
  $(G-S)^d$ in $V_i$ is at most $(\epsilon/K)|G|=\epsilon |V_i|$.
	
  \eqref{enum:sp2}$\Rightarrow$\eqref{enum:sp1} is trivial.
	
  \eqref{enum:sp1}$\Rightarrow$\eqref{enum:sp3}: The class
  $\mathscr C$ is nowhere dense as no ball of radius $d/2$ in $G-S$
  contains more than $\epsilon$-proportion of the vertices. (Otherwise
  many vertices have many vertices in their ball of radius $d$.)
\end{proof}

%\section{A non-regularity result for nowhere dense classes}
Answering a question of P.\ Simon, we now prove that nowhere dense
classes are not, in general, distal-defined. As distal-defined classes
have the strong Erd\H os-Hajnal property it will be sufficient to
prove that some nowhere dense class does not have this property. 

Let $\lambda_0\geq\lambda_1\geq\dots\geq\lambda_{n-1}$ be the
eigenvalues of the adjacency matrix of a graph~$G$. If $G$ is
connected and $d$-regular, the eigenvalues satisfy
$d=\lambda_{0}>\lambda_{1}\geq \dots \geq \lambda_{n-1}\geq -d$.  Let
$\lambda(G)=\max_{|\lambda_{i}|<d}|\lambda_{i}|$.  A $d$-regular graph
$G$ is a {\em Ramanujan graph} if $\lambda(G)\leq 2{\sqrt {d-1}}$.
Ramanujan graph are regular graphs with almost optimal spectral gap,
which are thus excellent expanders (see \cite{Lubotzky1988}).  Expanders have the property that
they behave like ``random graphs''; in particular, the number of edges
between two subsets $A$ and $B$ of vertices is close to the number
expected from the sizes of~$A$ and $B$ (see
e.g. \cite{hoory2006expander}).

\begin{lemma}[Expander mixing lemma]
\label{lem:exp}
Let $G$ be a $d$-regular graph on $n$ vertices with
$\lambda=\lambda(G)$. For any two subsets $S,T$ of vertices, let
$e(S,T)=\{(x,y)\in S\times T\mid \{x,y\}\in E(G)\}$. Then
\[
\left|e(S,T)-\frac{d|S||T|}{n}\right|<\lambda\sqrt{|S||T|(1-|S|/n)(1-|T|/n)}.
\]	
\end{lemma}

\begin{theorem}
  There exists a nowhere dense class $\mathscr R$ that does not have
  the strong Erd\H os-Hajnal property, hence it is not distal-defined.
\end{theorem}
\begin{proof}
  We consider here the construction of Ramanujan graphs due to
  Morgenstern \cite{MORGENSTERN199444}.  Let $q$ be an odd prime power
  and let $g(x)\in\mathbb F_q[x]$ be irreducible of even degree
  $d$. Morgenstern constructs a $(q+1)$-regular Ramanujan graph
  $\Omega_g$ of order $n=q^{3d}-q^d$ or $(q^{3d}-q^d)/2$, and girth at
  least $2/3\log_q n$. Consider $g(x)$ of degree $d> 2q$.  Then
  $n\geq q^{6q}$ and ${\rm girth}(\Omega_g)>3q$. Denote this graph by
  $G_q$. Let

$$\mathscr R=\{G_q\mid q\in\mathbb N\}.$$

\pagebreak
\begin{claim}
  The class $\mathscr R$ is nowhere dense.
\end{claim}
\begin{proof}
  Assume towards a contradiction that $\mathscr R$ is not nowhere
  dense. Then there exists an integer $p$ such that for every integer
  $n$ there is a graph $G_q\in\mathscr R$ that contains the
  $p$-subdivision of $K_n$ as a subgraph.  As triangles in $K_n$
  appear as cycles of length $3p+3$ in $G_q$ it follows that the girth
  of $G_q$ is at most $3p+3$ hence $q\leq p$. As $G_q$ is
  $(q+1)$-regular, it follows that all the vertices in the subgraph of
  $G_p$ (which is a $p$-subdivision of $K_n$) have degree at most
  $q+1$ hence $n\leq q+2\leq p+2$. Hence choosing $n>p+2$ leads to a
  contradiction.  \cqed\end{proof}

\begin{claim}
  The class $\mathscr R$ does not have the strong Erd\H os-Hajnal
  property.
\end{claim}
\begin{proof}
  Assume towards a contradiction that there exists $\delta>0$ such that 
  every graph $n$-vertex graph $G\in\mathscr R$ contains a homogenous pair $(A,B)$ with $\min(|A|,|B|)>\delta n$.
%  for
%  all subsets $X,Y$ of a graph $G\in\mathscr R$ there exist subsets
%  $X'\subset X$ and $Y'\subset Y$ of size $|X'|> \delta |X|$ and
%  $|Y'|> \delta |Y|$ such that $(X',Y')$ is E-homogenous.

  Let $q>4/\delta^2$
  and let $G=G_q$.
  Assume for contradiction that $G$ contains a homogenous pair $(A,B)$ with both parts of size at least $\delta n$.
%  Let $q>16/\delta^2$
%   and let $X,Y$ be a bipartition of the vertex set
%  of $G_q$ into two equal size halves, and let $X'\subset X$ and
%  $Y'\subset Y$ be subsets such that $|X'|>\delta |X|$ and
%  $|Y'|>\delta |Y|$. Let $n$ be the order of $G_q$. 
Let $z$ be the
  number of edges between $A$ and $B$.
  %$X'$ and $Y'$.  
  Then, according to the
  expander mixing lemma
  \smallskip
  \[
  \left|z-\frac{(q+1)|A||B|}{n}\right|<\lambda\sqrt{|A||B|(1-|A|/n)(1-|B|/n)}.
  \]
  \smallskip
%  \[
%  \left|z-\frac{(q+1)|X'||Y'|}{n}\right|<\lambda\sqrt{|X'||Y'|(1-|X'|/n)(1-|Y'|/n)}.
%  \]
%  By shrinking $X'$ and $Y'$ if necessary we can assume
%  $|X'|=|Y'|=\delta n/2$. Hence
%  \[
%  \left|z-(q+1)\delta^2n/4\right|<\lambda n\delta(1-\delta/2)/2.
%  \]

  By shrinking $A$ and $B$ if necessary we can assume
  $|A|=|B|=\delta n$. Hence
  \[
  \left|z-(q+1)\delta^2n\right|<\lambda n\delta(1-\delta).
  \]
  
  Hence $z>0$ if
  \[
  (q+1)\delta>\lambda (1-\delta),
  \]
%  Hence $z>0$ if
%  \[
%  (q+1)\delta/2>\lambda (1-\delta/2),
%  \]
%  i.e.\  $(q+1)/\lambda>(2/\delta-1)$. But
%  $(q+1)/\lambda\approx \sqrt{q}/2$.  As $q>16/\delta^2$ we deduce
%  that no pair $(X',Y')$ of subsets size $\delta n/2$ of $X$ and $Y$
%  can be E-homogeneous. (Note that they cannot indeed form a complete
%  bipartite subgraph as the degree are at most $q\ll \delta n/2$.
  i.e.\  $(q+1)/\lambda>(1/\delta-1)$. But
  $(q+1)/\lambda\approx \sqrt{q}/2$.  As $q>4/\delta^2$ we deduce
  that no pair $(A,B)$ of subsets size $\delta n$ 
  can be homogeneous. (Note that they cannot indeed form a complete
  bipartite subgraph as the degrees are at most $q\ll \delta n$.
  \cqed
\end{proof}

The theorem now directly follows from the previous two claims.
\end{proof}

We conclude with a negative result related to just constructed class $\mathscr R$:
\begin{corollary}
\label{cor:nd-noreg}
There is no $0<\epsilon<1$ and no integer $K$ such that for every $n$-vertex graph
$G\in\mathscr R$, there exists an $\epsilon$-nice partition $V=V_1\cup\dots\cup V_k$ of
$V(G)$ with $k \leq K$. 
\end{corollary}
\begin{proof}
	Let $G\in\mathscr R$ be an $n$-vertex graph. We prove that the existence of an $\epsilon$-nice partition of size $k$ would  imply 
	that $G$ contains a homogenous pair $(A,B)$ of size at least $\delta n$, where $\delta=(1-\epsilon)/k^2$, contradicting the fact that $\mathscr R$ does not have the strong Erd\H os-Hajnal property.
	 
		Assume $V_1,\dots,V_k$ is an $\epsilon$-nice partition, and let $\Sigma$ be the set of all pairs $(i,j)$ with $(V_i,V_j)$ homogenous. As the partition is $\epsilon$-nice we have
		$\sum_{(i,j)\in\Sigma}|V_i|\,|V_j|>(1-\epsilon)n^2$. It follows that there exists a pair $(i,j)\in\Sigma$ such that 
		$|V_i|\,|V_j|>\frac{1-\epsilon}{k^2}n^2$. Then either $i=j$ and by splitting $V_i$ into two parts one gets a homogenous pair $(A,B)$ of vertices, with each of $A$ and $B$ of size at least $\frac{\sqrt{1-\epsilon}}{2k}n$. Otherwise, $i\neq j$ and both $V_i$ and $V_j$ have size at least $\frac{1-\epsilon}{k^2}n$.
\end{proof}

%% file: conclusion.tex
\begin{figure}[ht]
\begin{center}
	\includegraphics[width=\textwidth]{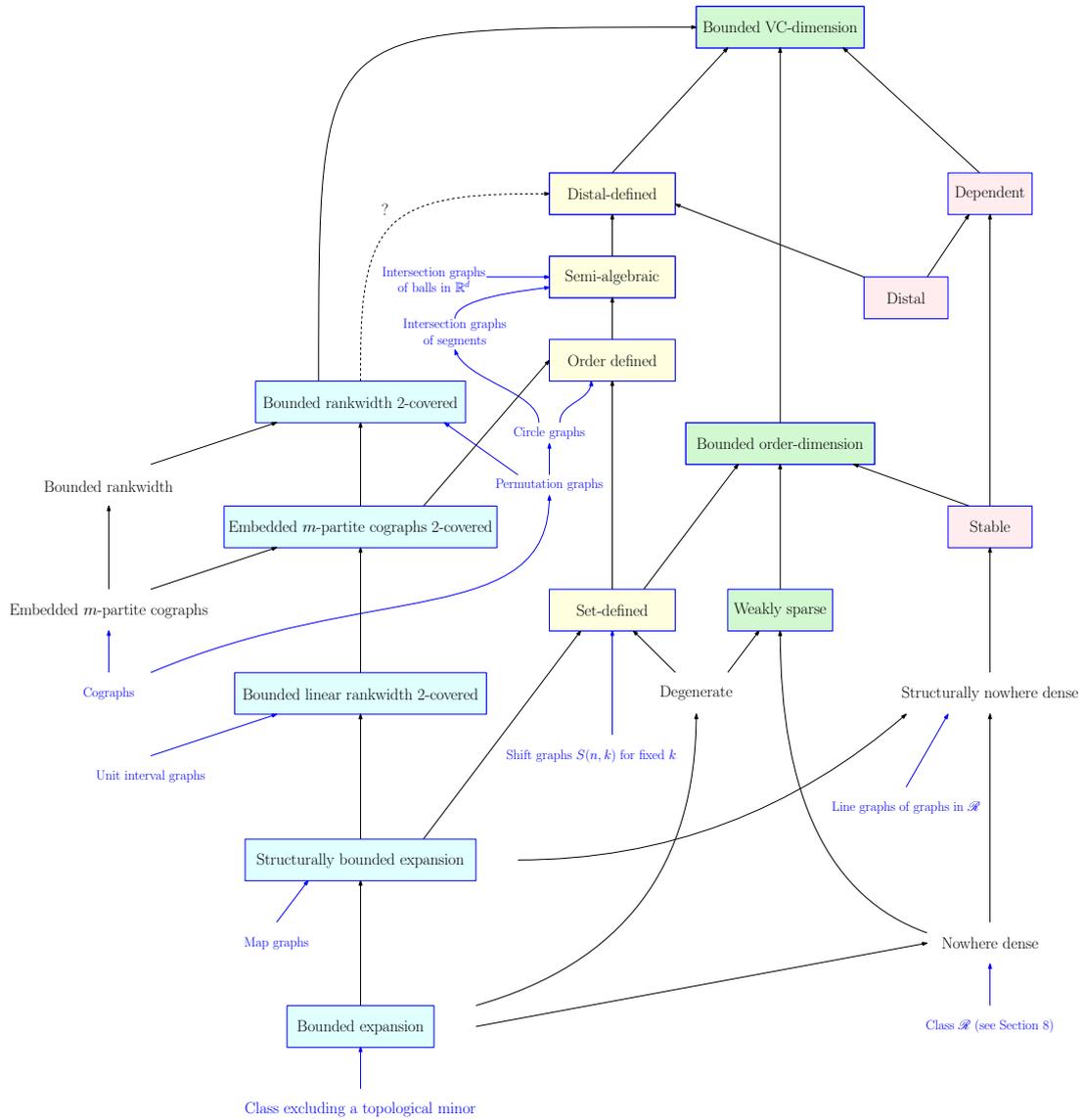}
\end{center}
\caption{Inclusion of graph classes. In blue boxes, the class defined by covers; in yellow boxes, classes defined from the age of a structure; in green boxes, classes defined by excluding a semi-induced bipartite graph; in pink boxes, model theoretical diving lines.}\label{fig:classes}
\end{figure}	 

\section*{Conclusion}

In this paper we surveyed and studied regularity properties of 
various graph classes. Our work highlights a strong and fruitful 
connection between graph theory and model theory. \Cref{fig:classes}
displays the studied concepts and examples of graph classes. 

We introduced the new notions
of order-defined and set-defined graph classes, which are special
cases of semi-algebraic graph classes. These classes nicely fit 
into a hierarchy of classes defined from the age of well-behaved 
infinite structures. Our study of structurally sparse graph classes
is mainly based on the notion of $2$-covers, and we showed that
covers nicely transport various properties of graph classes, including regularity lemmas. 
A natural follow-up of this work will be to consider hypergraphs and relational structures in full generality.
One major question that remained untouched in this paper are
algorithmic versions of the regularity lemmas and of $2$-covers.